\documentclass[10pt]{amsart}
\usepackage{amssymb,amsmath}
\usepackage{yfonts}
\usepackage[T1]{fontenc}

\newtheorem{thm}{Theorem}
\newtheorem{prop}[thm]{Proposition}

\newtheorem{defn}[thm]{Definition}

\newtheorem{cor}[thm]{Corollary}

\newtheorem{example}[thm]{Example}

\newcommand{\R}{\mathbb{R}}

\newcommand{\He}{\mathbb{H}}
\newcommand{\N}{\mathbb{N}}

\newcommand{\G}{\mathbb{G}}

\numberwithin{equation}{section}
\begin{document}
%\title[Hypoellipticity for systems in Carnot Groups]{Hypoellipticity for linear degenerate elliptic systems in Carnot Groups and
%applications}
\title{Horizontal Gauss Curvature Flow of Graphs in Carnot Groups}
\author{Erin Haller Martin}
\address{Westminster College\\
Fulton, MO 65251, USA } \email[Erin Haller Martin]{erin.martin@westminster-mo.edu}

\thanks{The results in this paper are part of
the author's Ph.D. Dissertation at the University of Arkansas and were partially funded by NSF grant NSF DMS-0134318.  I
would like to thank my advisor, L. Capogna, for suggesting the
horizontal Gauss curvature flow problem and for his advice, encouragement, and patience
along the way.}

%\subjclass[2000]{35J70,35E99}

\begin{abstract}
We show the existence of continuous viscosity solutions to the equation describing the flow of a graph in the Carnot group $\mathbb{G} \times \R$ according to its horizontal Gauss curvature. In doing so, we prove a comparison principle for degenerate parabolic equations of the form $u_t + F(D_0u, (D_0^2u)^*) = 0$ for $u$ defined on $\G$.  
\end{abstract}

\maketitle

\section{Introduction}

In the Euclidean setting, there has been extensive study of the evolution of surfaces by their Gauss curvature (see \cite{chow}, \cite{oliker}, \cite{urbas2}, \cite{urbas3}, \cite{andrews}, \cite{es:mc1}, \cite{firey}, \cite{tso}).  For a flow of surfaces parameterized by $x: M^n \times [0,T] \to \mathbb{R}^{n+1}$, such an evolution is described by 
\begin{equation}\label{curvatureflow}
\left\{ \begin{array}{c} \partial_tx(s,t) = -K(s,t)\vec{n}(s,t)\\ x(s,0) = x_0 \end{array}\right.
\end{equation}
where $K$ is the Gauss curvature of the surface $M_t = x(M^n,t)$, $\vec{n}(s,t)$ is the outer unit normal at $x(s,t)$ and $x_0$ describes the initial surface.  In 1974, W. J. Firey \cite{firey} proposed this flow as a model for the changing shape of a tumbling stone.  Assuming some existence and regularity  of solutions, he showed that surfaces which are convex and symmetric about the origin contract to points.  K. Tso \cite{tso} resolved the existence and regularity aspects by showing that $\eqref{curvatureflow}$ has a unique smooth solution for a maximal time interval $[0,T)$ when the initial surface is smooth, closed and uniformly convex.  Building on the results of Tso, B. Andrews \cite{andrews} generalized Firey's result to surfaces which are not necessarily symmetric but simply convex.  During this time, Y.-G. Chen, Y. Giga and S. Goto \cite{CGG}, and independently L. C. Evans and J. Spruck \cite{es:mc1}, developed a new approach to describing the evolution of surfaces which flow according to functions of their principal curvatures.  By considering the surfaces as level sets of a function, they could describe the flow using a scalar partial differential equation instead of the system of PDE which results from the parameterization described above.  Following this level set method, P. Marcati and M. Molinari \cite{marcati} reduced the problem of the Gauss curvature flow to showing the existence, uniqueness, and regularity of so-called viscosity solutions to a degenerate parabolic partial differential equation.  Such notions of solution provide for the existence and description of a solution past singularities which may develop in the surface flow.  

In recent years, viscosity theory has been extended to include solutions to equations defined on more general spaces such as the sub-Riemannian Carnot groups, each of which can be thought of as a limit of Riemannian manifolds (see \cite{bardi}, \cite{manfredi}, \cite{wang:comparison}, \cite{wang:convex} \cite{bieske}, \cite{bieske2}).  Curvature flows in this setting, in particular the mean curvature flow, have even been found to have applications to digital image reconstruction and neuroscience (\cite{cittisarti}, \cite{hp:min}).  
For the extension of the Gauss curvature flow problem to the setting of Carnot groups, the Riemannian Gauss curvature is substituted with the so-called horizontal Gauss curvature which is built by taking into account only the principal curvatures corresponding to horizontal tangent directions (see \cite{dgn3},\cite{cpt}).  In this setting, it is natural to begin with the case when the surface is given by the epigraph of a function $u:\mathbb{G} \to \mathbb{R}$, the evolution which is described by 
\begin{equation}\label{A}
\left\{\begin{array}{c}\partial_t u(p,t) = -K(p,t) \vec{n}_0(p,t) = \frac{\det \left((D_0^2u)^*\right)}{\left(\sqrt{1+|D_0u|^2}\right)^{m_1+1}} \\ u(p,0) = u_0(p)\end{array}\right.
\end{equation}
where $K(p,t)$ is the horizontal Gauss curvature, $D_0u$ is the horizontal gradient of $u$, $(D_0^2u)^*$ is the symmetrized horizontal Hessian of $u$, $\vec{n}_0$ is the normal to the surface projected onto the horizontal space $V^1$, $m_1$ is the dimension of $V^1$, and $u_0$ is used to describe the original surface.  Recall that the theory of viscosity solutions relies on the ellipticity (or coercivity) of $F(D_0u, (D_0^2u)^*) = -\frac{\det \left((D_0^2u)^*\right)}{\left(\sqrt{1+|D_0u|^2}\right)^{m_1+1}}$, i.e. $F(\eta, \mathcal{X}) \leq F(\eta,\mathcal{Y})$ whenever $\mathcal{X},\mathcal{Y} \in S^{m_1}(\R)$, $\mathcal{Y} \leq \mathcal{X}$ and $S^{m_1}(\R)$ denotes the set of all $m_1 \times m_1$ real symmetric matrices.  However, we see immediately that $F(\eta, \mathcal{X}) = -\frac{\det \mathcal{X}}{\left(\sqrt{1+|\eta|^2}\right)^{m_1+1}}$ only satisfies this property if $\mathcal{X}$ is positive definite.  Because of this, we will instead show the existence of continuous viscosity solutions to the following modified problem:
\begin{equation}
\left\{\begin{array}{c}\partial_t u(p,t) = \frac{\text{det}_+ \left((D_0^2u)^*\right)}{\left(\sqrt{1+|D_0u|^2}\right)^{m_1+1}} \\ u(p,0) = u_0(p)\end{array}\right.
\end{equation}
where $\text{det}_+ X = \prod_{i=1}^{m_1} \max\{\lambda_i,0\}$ and $\{\lambda_i\}$ denotes the eigenvalues of $X$.  In the Euclidean setting, it has been shown (see \cite{cei}, \cite{marcati}, \cite{tso}, \cite{andrews:cont}) that a strictly convex surface which flows according to its Gauss curvature remains strictly convex.  Thus, if the initial surface is strictly convex, the modified problem is equivalent to the original Gauss curvature flow problem.  In Section 4, we will show that the same is true for the Carnot group setting when we replace Euclidean convexity with the appropriate notion of convexity for Carnot groups, known as weak H-convexity (see Section 4.2 for details), and add some extra hypotheses concerning $u_0$.  

In recent years, some progress has been made in proving the existence and uniqueness of viscosity solutions to certain degenerate elliptic evolution equations in Carnot groups.  For example, existence and uniqueness results have been shown by L. Capogna and G. Citti \cite{CC} for the mean curvature flow equation.  T. Bieske \cite{bieske2} proved a comparison principle, and via Perron's method, existence and uniqueness results for solutions of degenerate parabolic equations which are defined on bounded domains in the Heisenberg group and satisfy certain uniform continuity conditions.  In \cite{manfredi}, F. Beatrous, T. Bieske, and J. Manfredi proved an analogous comparison principle for degenerate elliptic equations generated by vector fields.  As an exercise, and to more clearly present the main results of this paper, in Section 3 we will first use a combination of these methods to extend the results in \cite{bieske2} and \cite{manfredi} to degenerate parabolic equations on Carnot groups.  However, even these theorems will require the domain on which the sub- and supersolutions are defined to be bounded and $F$ to be admissible, i.e. $F$ must satisfy certain uniform continuity conditions.  We immediately see that because of this, these results cannot be applied to the horizontal Gauss curvature flow equation.  

Recently, C.-Y Wang \cite{wang:comparison} used a careful application of Jensen's maximum principle to the sup/inf convolutions of the sub- and supersolutions to prove a comparison principle for subelliptic equations on Carnot groups which do not necessarily satisfy any uniform continuity conditions.  However, he still required that the sub- and supersolutions be defined on bounded domains.  In the Euclidean setting, H. Ishii and T. Mikami \cite{im:level} used this same idea along with the additional requirement that the sub- and supersolutions possess certain growth conditions at infinity to prove a comparison principle for unbounded domains.  By combining the ideas in these two papers, in Section 3 we will prove a comparison principle that can be applied to sub- and supersolutions of the horizontal Gauss curvature flow equation defined on unbounded domains.  Combining this comparison principle with Perron's method will then yield the desired existence of continuous viscosity solutions.
%%%%%%%%%%%%%%%%%%%%%%%%%%%%%%%%%%%%%%%%%%%%%%%%%%%%%%%%%%%
\section{Carnot Groups}
%%%%%%%%%%%%%%%%%%%%%%%%%%%%%%%%%%%%%%%%%%%%%%%%%%%%%%%%%%%%%%%
In this section we introduce Carnot groups and summarize their basic properties.  
\begin{defn}
Let $\mathbb{G}$ be a Lie group and \text{\gothfamily{g}} its corresponding Lie algebra.  $\mathbb{G}$ is a stratified nilpotent Lie group of step $r \geq 1$ if \text{\gothfamily{g}} admits a vector space decomposition in $r$ layers

$$\text{\gothfamily{g}} = V^1 \oplus V^2 \oplus \cdots \oplus V^r$$
having the properties that $[V^1, V^j] = V^{j+1}$, $j=1,\ldots,r-1$ and $[V^j,V^r]=0$, $j=1, \ldots, r$.  
\end{defn}

Let $m_j = \text{dim}(V^j)$ and let $X_{i,j}$ denote a left-invariant basis of $V^j$ where $1 \leq j \leq r$ and $1 \leq i \leq m_j$.  The dimension of $\mathbb{G}$ as a manifold is $m=m_1+m_2+ \cdots +m_r$.  For simplicity, we will often set $X_i = X_{i,1}$.  We call the $\{X_{i}\}$ horizontal vector fields and call their span, denoted $H\mathbb{G}$, the horizontal bundle.  We call the $\{X_{i,j}\}_{2\leq j\leq r}$ vertical vector fields and call their span, denoted $V\mathbb{G}$, the vertical bundle.  Then $T\mathbb{G} = H\mathbb{G} \oplus V\mathbb{G}$.  We also define $n = m_2 + \cdots+ m_r$.

Let $g$ be a Riemannian metric on \text{\gothfamily{g}} with respect to which the $V^j$ are mutually orthogonal.  An absolutely continuous curve $\gamma:[0,1] \to \mathbb{G}$ is horizontal if the tangent vector $\gamma'(t)$ lies in $V^1$ for all $t$.  The Carnot-Carath\'eodory metric is then defined by 

$$d_{cc}(p,q) = \inf\int_0^1 \left(\sum_{i=1}^{m_1} \left \langle \gamma'(t),X_i|_{\gamma(t)}\right \rangle ^2_g \, dt\right)^{1/2},$$
where the infimum is taken over all horizontal curves $\gamma$ such that $\gamma(0) = p$, $\gamma(1) = q$ and $\left \langle \cdot,\cdot \right \rangle_g$ denotes the left invariant inner product on $V^1$ determined by $g$.
\begin{defn}
Let $\mathbb{G}$ be a simply connected Lie group with Lie algebra \text{\gothfamily{g}} and a Carnot Carath\'eodory metric $d_{CC}$ developed as above.  Then the pair $(\mathbb{G}, d_{CC})$ is a Carnot group.
\end{defn}

\textit{Note.}  By a standard abuse of notation, we will refer to $\mathbb{G}$ as the Carnot group, implying its association with $d_{CC}$.
\\

As the exponential map exp:\hspace{3pt}\text{\gothfamily{g}} $\mapsto \mathbb{G}$ is a global diffeomorphism, we can use exponential coordinates on $\G$.  In this way, a point $p \in \mathbb{G}$ has coordinates $p_{i,j}$ for $1 \leq i \leq m_j$, $1 \leq j \leq r$ if 

$$p = \text{exp}\left(\sum_{j=1}^r \sum_{i=1}^{m_j}p_{i,j}X_{i,j}\right).$$
In this setting, the non-isotropic dilations are the group homomorphisms given by 
$$\delta_s \left(\sum_{j=1}^r \sum_{i=1}^{m_j}p_{i,j}X_{i,j}\right)=\sum_{j=1}^r \sum_{i=1}^{m_j}s^j p_{i,j}X_{i,j},$$
where $s > 0$.

Using these coordinates an equivalent distance on $\mathbb{G}$ is the gauge norm given by 
$$|p|_g = \left(\sum_{j=1}^r\left(\sum_{i=1}^{m_j} |p_{i,j}|^2\right)^{\frac{r!}{j}}\right)^{\frac{1}{2r!}}.$$
(Note that typically $|\cdot|_g$ is only a quasinorm rather than a true norm, i.e. the inequality $|p\,q|_g \leq |p|_g|q|_g$ must be replaced with $|p\,q|_g \leq C|p|_g|q|_g$ for some constant $C < \infty$.)
Then we have 
$$d_g(p,q) = |q^{-1}p|_g.$$
Using this distance we define the gauge balls $B(p,r) = \{ x\in \G\hspace{3pt} |\hspace{3pt} d_g(p,x) < r\}$.  We also have $|B(p,r)| = w_Gr^Q$ where $|B(e,1)| = w_G$, $e$ is the group identity and $Q = \sum_{k=1}^r km_k$ is the so-called homogeneous dimension of $\G$ (see \cite{fol:1975}). 
\setcounter{thm}{0}
\begin{example}\label{example0} \textnormal{\textbf{Euclidean Space} $\mathbb{E}^n$}\\
The simplest example of a Carnot group is Euclidean space, $\mathbb{E}^n = (\mathbb{R}^n, |\cdot|\hspace{2pt})$, which is a Carnot group of step 1. 
\end{example}
\begin{example}\label{example1} \textnormal{\textbf{The Heisenberg Group} $\He^n$}\\
The simplest example of a non-abelian Carnot group is the Heisenberg
group, ${\mathbb{H}}^n$, which is a Carnot group of step $2$.  It is
the Lie group with underlying manifold $\mathbb{R}^{2n}\times
\mathbb{R}$ endowed with the non-commutative group law
$$(x,x_{2n+1})(x',x_{2n+1}')=(x+x',x_{2n+1}+x'_{2n+1}+2[x,x'])\,,$$ where $x,x' \in
\mathbb{R}^{2n}$, $x_{2n+1},x_{2n+1}' \in \mathbb{R}$, and
$[x,x']=\sum_{i=1}^n({x'}_i x_{n+i}-x_i{x'}_{n+i})$. The vector
fields $X_{i,1} =
\partial_{x_i} - \frac{1}{2}x_{n+i}\partial_{x_{2n+1}}$, $X_{i+n,1} =
\partial_{x_{n+i}} + \frac{1}{2}x_i \partial_{x_{2n+1}}$, for $i=1, \ldots ,n$ and
$X_{1,2}=\partial_{x_{2n+1}}$ form a left-invariant vector basis for the Lie
algebra of $\mathbb{H}^n$.  Its Lie algebra \text{\gothfamily{h}} can be written as the
vector sum $\text{\gothfamily{h}}=V^1 \oplus V^2$, where $V^1 = \text{Span}\{X_{1,1}, \ldots,
X_{2n,1}\}$ and $V^2 = \text{Span}\{X_{1,2}\}$.
\end{example}

\begin{example}\label{example2} \textnormal{\textbf{H-type Groups}}\\
A Carnot group $\G$ is said to be of Heisenberg type, or of H-type, if the Lie algebra $\text{\gothfamily{g}}$ is of step two with $\text{\gothfamily{g}} = V^1\oplus V^2$ and if there is an inner product $\langle \cdot , \cdot \rangle$ on $\text{\gothfamily{g}}$ such that the linear map $J: V^1 \to \text{End}\,V^2$ defined by the condition 
\begin{equation}\label{htype1}
\left \langle J_z(u),v\right\rangle = \left \langle z, [u,v]\right\rangle
\end{equation}
satisfies 
\begin{equation}\label{htype2}
J_z^2 = -|z|^2\text{Id}
\end{equation}
for all $z \in V^2$.  The following are consequences of $\eqref{htype1}$ and $\eqref{htype2}:$
\begin{equation}\label{htype3}
|J_z(v)| = |z||v|
\end{equation}
\begin{equation}\label{htype4}
\left\langle J_z(v),v\right\rangle = 0.
\end{equation}
Such groups were introduced by Kaplan in \cite{kaplan}.  For more information, we refer the reader to \cite{kaplan} and \cite{HH}.
\end{example}
\setcounter{thm}{2}
We will also need the following definition of spaces of continuous functions.
\begin{defn}
Let $\G$ be a Carnot group.  For $j,k,l \in \N$, $\Omega \subset \G$, and $T > 0$ we let $\mathcal{C}_j^{k,l}(\Omega \times (0,T))$ represent the set of functions $f : \Omega \times [0,T] \to \R$ such that the components of $f$ as well as all of the horizontal derivatives up to order $k$, all of the derivatives along the second layer up to order $j$, and all of the time derivatives up to order $l$ of the components of $f$ are continuous in $\Omega \times (0,T)$.
\end{defn}
We will also use $C^{k,l}$ in the usual way to denote the set of functions $f: \Omega \times [0,T] \to \R$ having the property that the components of $f$ as well as all of the Euclidean spatial derivatives up to order $k$ and all of the time derivatives up to order $l$ are continuous.  

Finally we will need the following notation concerning the derivatives of $u$.  
Given a function $u:\mathbb{G} \times [0,T] \to \R$ we consider the (full) spatial gradient of $u$ given by 

$$D_{\text{\gothfamily{X}}}u= (X_{i,j}u)_{1\leq i \leq m_j,1\leq j \leq r} \in \R^m.$$
As a vector field, this is written as 
$$D_{\text{\gothfamily{X}}}u = \sum_{j=1}^r \sum_{i=1}^{m_j}(X_{i,j}u)X_{i,j}.$$
The horizontal gradient of $u$ is

$$D_0u = (X_{i}u)_{1\leq i \leq m_1} \in \R^{m_1},$$
or as a vector field
$$D_0u = \sum_{i=1}^{m_1}(X_{i}u)X_{i}.$$
We will also write 

$$D_1u = (X_{i,2}u)_{1 \leq i \leq m_2}$$
for the gradient along the second layer,
$$D_0^2u = (X_{i}X_{j}u)_{1 \leq i \leq m_1}$$
for the second order derivatives corresponding to $V^1$, and $(D_0^2u)^*$ for its symmetric part $\frac{1}{2}(D_0^2u+(D^2_0)^T)$, where $A^T$ denotes the transpose of $A$.  We will also always let $S^{m_1}(\R)$ denote the set of all $m_1 \times m_1$ real symmetric matrices.

%%%%%%%%%%%%%%%%%%%%%%%%%%%%%%%%%%%%%%%%%%%%%%%%%%%%%%%%%%%%%%%%%%%%%%%%%%%%%%%%%%%%%%%%%%%%%%%%%%%%%%%%%%%
\section{Viscosity Solutions and Comparison Principles}
%%%%%%%%%%%%%%%%%%%%%%%%%%%%%%%%%%%%%%%%%%%%%%%%%%%%%%%%%%%%%%%%%%%%%%%%%%%%%%%%%%%%%%%%%%%%%%%%%%%%%%%%%%%%%
In this section we first define the notion of viscosity solutions used in the main portion of this paper.  Note that other equivalent definitions exist but are not appropriate for the proofs presented here.  We refer the reader to \cite{hallerE}, as well as \cite{manfredi}, \cite{manfredinotes}, \cite{bieske2}, for more details.

We will consider parabolic equations of the form
\begin{equation}\label{main}
u_t + F(t,p,u,D_0u,D_1u,(D^2_0u)^*) = 0
\end{equation}
for continuous $F:[0,T] \times \mathbb{G} \times \R \times \R^{m_1} \times \R^{m_2} \times S^{m_1}(\R) \to \R$.
\begin{defn}
Let $\G$ be a Carnot group, $\mathcal{O} \subset \G$ an open set and $\mathcal{O}_T = \mathcal{O} \times (0,T)$.  Let $(p_0,t_0) \in \mathcal{O}_T$.  A lower semicontinuous function $v$ is a viscosity supersolution of the equation $\eqref{main}$ in $\mathcal{O}_T$ if for all $\varphi \in \mathcal{C}^{2,1}_1(\mathcal{O}_T)$ such that $u-\varphi$ has a local minimum at $(p_0,t_0)$ one has 
$$\varphi_t(p_0,t_0) + F(t_0,p_0,u(p_0,t_0),D_0\varphi(p_0,t_0),D_1\varphi(p_0,t_0),(D_0^2\varphi)^*(p_0,t_0)) \geq 0.$$
\end{defn}
\begin{defn}
Let $\G$ be a Carnot group, $\mathcal{O} \subset \G$ an open set and $\mathcal{O}_T = \mathcal{O} \times (0,T)$.  Let $(p_0,t_0) \in \mathcal{O}_T$.  An upper semicontinuous function $u$ is a viscosity subsolution of the equation $\eqref{main}$ in $\mathcal{O}_T$ if for all $\varphi \in \mathcal{C}^{2,1}_1(\mathcal{O}_T)$ such that $u-\varphi$ has a local maximum at $(p_0,t_0)$ one has 
$$\varphi_t(p_0,t_0) + F(t_0,p_0,u(p_0,t_0),D_0\varphi(p_0,t_0),D_1\varphi(p_0,t_0),(D_0^2\varphi)^*(p_0,t_0)) \geq 0.$$
\end{defn}

\begin{defn}
A function $u$ is a viscosity solution of $\eqref{main}$ if 
$$u^*(p,t):= \lim_{r \downarrow 0}\sup \{u(q,s) : |q^{-1}p|_g + |s-t| \leq r\}$$
is a viscosity subsolution and 
$$u_*(p,t):= \lim_{r \downarrow 0}\inf \{u(q,s) : |q^{-1}p|_g + |s-t| \leq r\}$$
is a viscosity supersolution.
\end{defn}

As in the Euclidean setting, each of the above definitions has an equivalent form stated in terms of parabolic semi-jets.  More details can be found in \cite{hallerE}.

To proceed, we need the following definitions.
\begin{defn}
Let $\G$ be a Carnot group and $\mathcal{O} \subset \G$ an open set.  A continuous function 
$$F: [0,T] \times \bar{\mathcal{O}} \times \R \times \R^{m_1} \times \R^{m_2} \times S^{m_1}(\R) \to \R$$
is degenerate elliptic if 
$$F(t,p,r,\eta,\xi,\mathcal{X}) \leq F(t,p,s,\eta,\xi,\mathcal{Y})$$
whenever $\mathcal{Y} \leq \mathcal{X}$, $r \leq s$.
\end{defn}
\noindent \textit{Note.} We will call the equation $u_t(p,t) + F(t,p,r,\eta,\xi,\mathcal{X}) = 0$ degenerate parabolic if $F$ is degenerate elliptic. 

\begin{defn}
Let $\G$ be a Carnot group and $\mathcal{O} \subset \G$ an open set.  A degenerate elliptic function
$$F: [0,T] \times \bar{\mathcal{O}} \times \R \times \R^{m_1} \times \R^{m_2} \times S^{m_1}(\R) \to \R$$
is admissible if there exist $\sigma > 0$ and $\omega_i : [0,\infty] \to [0,\infty]$ with $\omega_i(0+)=0$ so that $F$ satisfies
\begin{eqnarray*}
|F(t,p,r,\eta,\xi,\mathcal{X}) - F(t,q,r,\eta,\xi,\mathcal{X})| &\leq& \omega_1(|q^{-1}p|_g) \\
|F(t,p,r,\eta^+, \xi, \mathcal{X} - F(t,p,r,\eta^-,\xi,\mathcal{X})| &\leq& \omega_2(|\eta^+ -\eta^-|)\\
|F(t,p,r,\eta,\xi^+, \mathcal{X}) - F(t,p,r,\eta,\xi^-,\mathcal{X})| &\leq& \omega_3(|\xi^+ - \xi^-|)\\
|F(t,p,r,\eta,\xi,\mathcal{X}) - F(t,p,r,\eta,\xi,\mathcal{Y})| &\leq& \omega_4(\|\mathcal{X} - \mathcal{Y}\|)\\
\end{eqnarray*}
for each fixed $t$ and where $\|B\| = \sup \left\{ |\lambda| : \lambda \text{ is an eigenvalue of } B\right\}$.

\end{defn}

In \cite{manfredi}, Beatrous, Bieske, and Manfredi prove the general vector field analogue of the Euclidean comparison principle (\cite[Theorem 3.3]{usersguide}) for admissible PDE  on bounded domains.  In other words, they prove a comparison principle for when the vector fields $\{\partial_{x_1},\ldots, \partial_{x_m}\}$ are replaced by an arbitrary collection of smooth vector fields $\text{\gothfamily{X}}$  and the equation $F(x,u(x),Du,D^2u) = 0$ is replaced by 
$$F(x,u(x),D_{\text{\gothfamily{X}}}u,(D^2_{\text{\gothfamily{X}}}u)^*) = 0$$ 
where $F$ is admissible.  Manfredi \cite{manfredinotes} showed that such a comparison principle still holds when we consider instead 
$$F(x,u(x),D_0u, D_1u, (D_0^2u)^*) = 0.$$  
In \cite{bieske2}, Bieske extended Manfredi's work to parabolic equations of the form 
$$u_t + F(t,p,u,D_0u,(D_0^2u)^*) = 0$$
for $u:\Omega \times (0,T) \to \R$, $\Omega \subset \He^n$ and $F$ admissible.
It is straightforward to extend these proofs to the case when the equation is degenerate parabolic and $u : \Omega \times (0,T) \to \R$, $\Omega \subset \G$ bounded.  Following the proofs in \cite{manfredi}, \cite{manfredinotes} ,\cite{bieske}, and \cite{bieske2},  we use the Euclidean parabolic maximum principle to obtain Euclidean parabolic semi-jets and then "twist" and restrict them appropriately to obtain subriemannian parabolic semi-jets and thus viscosity sub/super solutions.  For the full details of the proof of this theorem, Theorem \ref{comparison}, please see \cite{hallerE}.  Note that this proof requires the definitions of viscosity sub/super solutions to be given in terms of semi-jets but that these definitions are equivalent to the ones given above.

\setcounter{thm}{0}
\begin{thm}\label{comparison}
Let $\G$ be a Carnot group, $\mathcal{O} \subset \G$ be a bounded domain, $\mathcal{O}_T = \mathcal{O} \times (0,T)$ with $T > 0$, and $\psi \in C(\bar{\mathcal{O}})$.  If $u$ is a viscosity subsolution and $v$ is a viscosity supersolution to \begin{equation}\label{problem}
\left\{\begin{array}{ll}
\text{(E)} & u_t + F(t,p,u,D_0u,D_1u,(D^2_0u)^*) = 0 \hspace{5pt} \text{ in } \mathcal{O}_T,\\
\text{(BC)} & u(p,t) = 0 \text{ for } 0 \leq t < T \text{ and } p \in \partial \mathcal{O},\\
\text{(IC)} & u(p,0) = \psi(p) \text{ for } p \in \overline{\mathcal{O}}
\end{array}\right.
\end{equation}
where $F$ is admissible, then $u \leq v$ on $\mathcal{O} \times [0,T)$.  Here, as in \cite{usersguide}, by a viscosity subsolution to $\eqref{problem}$ on $\overline{\mathcal{O}} \times [0,T)$ we mean a function $u$ that is a viscosity subsolution to (E) such that $u(p,t) \leq 0$ for $0 \leq t < T$, $p \in \partial \mathcal{O}$ and $u(p,0) \leq \psi(p)$ for $p \in \overline{\mathcal{O}}$.  We define the notions of supersolution in the same manner.
\end{thm}

%-------------------------------------------------------------------------------------------------------

For the Gauss curvature flow equation, however, it is necessary to consider a function $F$ which is not admissible.  In particular, 
\begin{equation}\label{general}
 u_t + F(D_0 u , (D_0^2 u)^* ) = u_t - \frac{\det \left((D_0^2u)^*\right)}{\left(\sqrt{1+|D_0u|^2}\right)^{m_1+1}}= 0
\end{equation}
where $F$ does not satisfy 
$$ |F(\eta, \mathcal{X}) - F(\eta, \mathcal{Y})| \leq \omega(\|\mathcal{X} - \mathcal{Y}\|).$$ Further, we would like to consider PDE which are defined on all of $\G$, not just those defined on a bounded domain $\Omega \subset \mathbb{G}$.  However, the technique of constructing the subriemannian parabolic jets from Euclidean jets cannot be extended to include this situation.  The Euclidean parabolic comparison principle (see \cite[Theorem 5.1, Theorem 8.3]{usersguide}) yields $\mathcal{X}, \mathcal{Y} \in S^{m}$ such that $\mathcal{X}$ and $\mathcal{Y}$ are the second order parts of the Euclidean sub- and superjets, respectively, and 
$$\left(\begin{array}{cc} \mathcal{X} & 0 \\ 0 & \mathcal{-Y} \end{array}\right) \leq 3 \alpha \left(\begin{array}{cc} I & -I \\ -I & I \end{array}\right) + 3\epsilon I$$
for $\alpha$ large and $\epsilon$ small.  "Twisting" the matrix $3 \epsilon I$ poses a problem because the domain is unbounded. As the elements of the change of basis matrix which takes $\{ \partial_{x_1}, \ldots, \partial_{x_m}\}$ to $\text{\gothfamily{X}}$ are unbounded, the right hand side of the "twisted" form of the above inequality may go to infinity despite letting $\epsilon \to 0$. As the proof of Theorem \ref{comparison} relies heavily on the right hand side remaining bounded for each fixed $\alpha$ as $\epsilon \to 0$, the proof fails in this case.

To resolve these issues we use a different approach to the proof. In \cite{wang:comparison}, Wang proves a comparison principle for fully non-linear subelliptic equations which are defined on a bounded domain but not necessarily admissible.  The strategy for the proof relies first on the existence of a strict viscosity supersolution to $\eqref{general}$ which can be constructed by perturbing any given viscosity supersolution (\cite[Lemma 2.1]{wang:comparison}).  The sup convolution of the viscosity subsolution $u$ and the inf convolution of the strict viscosity supersolution $v$ (\cite[Definition 3.1]{wang:comparison}), denoted $u^{\epsilon}$ and $v_{\epsilon}$ respectively, are then used to construct functions $w^\epsilon(u^\epsilon,v_\epsilon)$ to which Jensen's maximum principle (\cite[Lemma 3.10]{jensen},\cite[Lemma 3.4]{CGG}) is applied.  This yields the existence of a sequence of points at which the functions $w^\epsilon$ are in $\mathcal{C}^{2,1}_1$ and have a local maximum. Using the relationships between the derivatives of $u^\epsilon$ and the derivatives of $v_\epsilon$ at the local maxima along with the fact that $u^\epsilon$ is a subsolution and $v_\epsilon$ is a strict supersolution (\cite[Proposition 3.3]{wang:comparison}) yields the result after taking appropriate limits.    These constructions allow for the lack of admissibility of $F$, though they still require that $F$ be degenerate elliptic.  For the parabolic proof, we follow the same overall strategy except we will perturb the given subsolution to construct a strict subsolution.

To extend this proof to the case when $\Omega = \mathbb{G}$ we need $u$ and $v$ to possess appropriate bounds at infinity.  For this we adapt the proof of \cite[Theorem 1]{im:main} to the Carnot group case and prove the following:

\begin{thm}\label{comparisongauss}
Let $\G$ be a Carnot group.  Suppose that $F: \R^{m_1} \times S^{m_1}(\R) \to \R$ is continuous, degenerate elliptic and is such that if $u(x,t)$ is a subsolution then $\mu u(x,\theta t )$  is a subsolution of \begin{equation}\label{starstar}u_t + F(D_0u, (D_0^2u)^*) = 0\end{equation} for $\theta \in (0,1)$ , $\mu \in (0,1)$ satisfying $\theta \mu^{-(m_1 - 1)} \leq 1$.
Let $h_0: \mathbb{G} \to \R$ be such that $h_0 \in C(\mathbb{G})$ and 
\begin{equation}\label{h0bound}
h_0(x) \geq \epsilon_0 |x|_g^{2r!} \hspace{5pt} \forall x \in \mathbb{G}
\end{equation}
for some $\epsilon_0 > 0$.  Suppose $u$ is a viscosity subsolution and $v$ is a viscosity supersolution to $\eqref{starstar}$ on $\mathbb{G} \times [0,\infty)$ such that 
$$u(x,0) \leq v(x,0) \hspace{5pt} \forall x \in \mathbb{G}$$
and that for each $T > 0$
\begin{equation}\label{h0bound1}
\sup_{(x,t) \in \mathbb{G} \times [0,T]} \left(|u(x,t) - h_0(x)| + |v(x,t) - h_0(x)|\right) < \infty.
\end{equation}
Then 
\begin{enumerate}
\item [(i)] for any $\theta \in (0,1)$ the inequality $u(x,\theta t) \leq v(x,t)$ holds for all $(x,t) \in \mathbb{G} \times (0,\infty)$
\item [(ii)]if we assume instead that $u$ is continuous in $t$ then $u \leq v$ on $\mathbb{G} \times [0,\infty)$
\item [(iii)] or for more general viscosity solutions, if we assume that $h_0 \in C^2(\mathbb{G})$ and 
$$\text{det}_+ D^2_0 h_0(x) \leq C(1+|D_0h_0(x)|^2)^{(m_1+1)/2} \hspace{5pt} \forall x \in \mathbb{G}$$
for some constant $C>0$ and that for each $\epsilon > 0$ there exists a constant $R = R(\epsilon) > 0$ such that for all $x\in \mathbb{G}$, if $|x|_g \geq R$ then 
$$u(x,0) - \epsilon \leq h_0(x) \leq v(x,0) + \epsilon,$$
then $u \leq v$ on $\mathbb{G} \times [0,\infty).$
\end{enumerate}
\end{thm}

\begin{proof}
(i)  Fix $\theta \in (0,1)$ and $T > 0$.  We need to show that 
\begin{equation}\label{theta1}
u(x,\theta t) \leq v(x,t) \hspace{5pt} \forall (x,t) \in \mathbb{G} \times [0,T).
\end{equation}
Define $$u_\theta(x,t) := u(x,\theta t).$$  
Let $\mu \in (0,1)$ be such that $\theta\mu^{-(m_1 - 1)} \leq 1$.  
By assumption $\mu u_\theta$ is also a viscosity subsolution.  In order to show $\eqref{theta1}$ it is enough to show that for all $\mu \in (0,1)$ such that $\theta \mu^{-(m_1-1)} \leq 1$ we have 
\begin{equation}\label{theta2}
\mu u_\theta \leq v \text{ on } \mathbb{G} \times [0,T)
\end{equation}
By $\eqref{h0bound}$, there exists $C_0>0$ such that 
$$|u(x,t) - h_0(x)| + |v(x,t) -h_0(x)| \leq C_0 \hspace{5pt} \forall x \in \mathbb{G} \times [0,T).$$
From this we have that 
\begin{eqnarray*}
\mu u_\theta(x,t) &\leq& \mu (h_0(x) + C_0)\\
&\leq& v(x,t) - h_0(x) + C_0 + \mu (h_0(x) + C_0) \\
&\leq& v(x,t) + (1+\mu)C_0 - (1-\mu)\epsilon_0 |x|_g\\
&\leq& v(x,t) - (1-\mu)\epsilon_0 |x|_g^{2r!} + 2C_0.
\end{eqnarray*}
Therefore, there exists $R> 0$ such that if $|x|_g^{2r!} \geq R^{2r!}$ then 
$$v(x,t)- (1-\mu)\epsilon_0 |x|_g^{2r!} + 2C_0 \leq v(x,t).$$
Thus
$$\mu u_\theta(x,t) \leq v(x,t) \text{ on } (\mathbb{G} \backslash B(0,R)) \times [0,T).$$
Now it remains to be shown that 
$$\mu u_\theta(x,t) \leq v(x,t) \text{ on } B(0,R) \times [0,T).$$
To obtain a contradiction, suppose that 
\begin{equation}
\sup_{\overline{B(0,R)} \times [0,T)} (\mu u_\theta - v) > 0.
\end{equation}
To simplify the notation, from now on we will write 
$$w = \mu u_\theta.$$
Consider $\tilde{w} = w - \frac{\epsilon}{T-t}$ instead of $w$.  Then $\tilde{w}$ is a viscosity subsolution and we recall that this implies that if we can show our claim for 
$$\left\{ \begin{array}{l}
w_t + F(D_0w, (D_0^2w)^*) \leq -\delta < 0 \\
\lim_{t \to T} w(x,t) = -\infty \text{ uniformly on } \overline{B(0,R)} \end{array} \right.$$
taking the limit as $\delta \to 0$ will yield the desired result. \\
Since we have that 
$$w \leq v \text{ on } \partial B(0,R) \times [0,T)$$
and 
$$\sup_{\overline{B(0,R)} \times [0,T)} (w - v) > 0$$
we must have 
$$\max_{\partial B(0,R) \times [0,T)} (w - v) \leq 0 < \sup_{B(0,R) \times [0,T)} (w - v).$$  Further, recall that $$u(x,0) \leq v(x,0).$$  Thus by the definition of $w$, $$w(x,0) \leq v(x,0).$$  Therefore, we actually have 
$$\max_{\partial B(0,R) \times [0,T)} (w - v) \leq 0 < \sup_{B(0,R) \times (0,T)} (w - v).$$
Let $\epsilon > 0$ and define for $(x,t) \in \overline{B(0,R)} \times [0,T)$
$$w^{\epsilon}(x,t) = \sup_{(y,s) \in \overline{B(0,R)} \times [0,T)} \left\{w(y,s) - \frac{1}{2\epsilon}(|y^{-1}x|_g^{2r!} + |t-s|^2)\right\}$$
and 
$$v_{\epsilon}(x,t) = \inf_{(y,s) \in \overline{B(0,R)} \times [0,T)} \left\{v(y,s) + \frac{1}{2\epsilon}(|y^{-1}x|_g^{2r!} + |t-s|^2)\right\}.$$
By \cite[Propsition 2.3]{wang:convex} we have that $w^{\epsilon},-v_{\epsilon}$ are semi-convex, $w^{\epsilon},-v_{\epsilon}$ are Lipschitz with respect to $|\cdot|_g + |\cdot|$, and $w^\epsilon \to w$, $v_\epsilon \to v$ pointwise as $\epsilon \to 0$.
Further, if we define $R_0 = \max \{\|w\|_{L^\infty(B(0,R) \times (0,T))}, \|v\|_{L^\infty(B(0,R) \times (0,T))}\}$ (which we know exists by $\eqref{h0bound1}$) and  
$$B(0,R)_{(1+2R_0)\epsilon} = \{x \in B(0,R) : \inf_{y \in \partial B(0,R)}:|y^{-1}x|_g^{2r!} \geq (1+2R_0)\epsilon\}$$
$$[0,T)_{(1+2R_0)\epsilon} = \{t \in [0,T) : \inf_{s \in \{0,T\}}:|t-s|^2 \geq (1+2R_0)\epsilon\}$$  
then we have that $w^\epsilon, v_\epsilon$ are viscosity sub- and supersolutions, respectively, on $B(0,R)_{(1+2R_0)\epsilon} \times [0,T)_{(1+2R_0)\epsilon}$.  
Now let $\partial B_T = (\partial B(0,R) \times [0,T]) \cup (B(0,R) \times \{0\})$ and recall 
$$\max_{\partial B_T} (w- v) \leq 0 < \alpha \leq \sup_{B(0,R) \times (0,T)} (w-v).$$

Therefore, there exists $(\overline{x},\overline{t}\,) \in B(0,R) \times (0,T)$ such that
$$\sup_{B(0,R) \times (0,T)}(w-v) = (w-v)(\overline{x},\overline{t}\,) \geq (w-v)(x,t) \hspace{5pt} \forall (x,t) \in \overline{B(0,R)} \times [0,T).$$
Then for each $\sigma > 0$, 
\begin{eqnarray*}
(w-v)(\overline{x},\overline{t}\,) \!\!\!&-&\!\!\! \frac{1}{\sigma} \left(|\overline{x}\,^{-1} \cdot \overline{x}|_g^{2r!} + |\overline{t}-\overline{t}\,|^2\right) \\
&=& (w-v)(\overline{x},\overline{t}\,)\\
&\geq& (w-v)(x,t) \\
&> & (w-v)(x,t) -\frac{1}{\sigma} \left(|x^{-1} \cdot \overline{x}|_g + |t-\overline{t}\,|\right) 
\end{eqnarray*}
for $(x,t) \neq (\overline{x},\overline{t}\,)$. Therefore, $(w-v)(x,t) -\frac{1}{\sigma} \left(|\overline{x}\,^{-1} \cdot x|_g + |t-\overline{t}\,|\right)$ has a strict maximum at $(\overline{x},\overline{t}\,)$.  Further, for $\sigma$ large enough, 
$$\sup_{B(0,R)\times (0,T)}\left((w-v)(x,t) -\frac{1}{\sigma} \left(|\overline{x}\,^{-1} \cdot x|_g^{2r!} + |t-\overline{t}\,|^2\right)\right) \geq \frac{\alpha}{2} > 0$$

Since $w^\epsilon, -v_\epsilon$ are upper semicontinuous and $w^\epsilon \to w$, $v_\epsilon \to v$ pointwise as $\epsilon \to 0$, we have that 
$$\Big(w^\epsilon - v_\epsilon  - \frac{1}{\sigma} \left(|\overline{x}\,^{-1}\! \cdot\! x|_g^{2r!} + |t-\overline{t}\,|^2\right) - \max_{\partial B_T} \big((w-v)-\frac{1}{\sigma} \left(|\overline{x}\,^{-1} \!\cdot\! x|_g^{2r!} + |t-\overline{t}\,|^2\right)\big)\Big)_+$$
converges monotonically to $0$ pointwise on $\partial B_T$ as $\epsilon \to 0$ where \\$f_+(x,t) = \max \{f(x,t),0\}$.  Therefore, by Dini's Lemma, the above convergence is uniform.  Then for $\epsilon$ small enough and $\sigma$ large enough, 
\begin{eqnarray*}
\max_{\partial B_T} \Big(w^\epsilon - v_{\epsilon} \!\!\! &-&\!\!\! \frac{1}{\sigma} \left(|\overline{x}\,^{-1}\! \cdot\! x|_g^{2r!} + |t-\overline{t}\,|^2\right)\Big)\\&\leq& 0 \\ &<& \frac{\alpha}{2} \\ &\leq& \sup_{B(0,R) \times (0,T)} \left (w -v - \frac{1}{\sigma} \left(|\overline{x}\,^{-1}\! \cdot\! x|_g^{2r!} + |t-\overline{t}\,|^2\right)\right)\\ &\leq& \sup_{B(0,R) \times (0,T)} \left(w^\epsilon - v_\epsilon - \frac{1}{\sigma} \left(|\overline{x}\,^{-1}\! \cdot\! x|_g^{2r!} + |t-\overline{t}\,|^2\right)\right).
\end{eqnarray*}
For the moment, fix $\sigma$ as large as necessary.  Thus we have that 
\begin{eqnarray*}
\max_{\overline{B(0,R)} \times [0,T)} \Big(w^\epsilon - v_\epsilon\!\!\!&-&\!\!\! \frac{1}{\sigma} \left(|\overline{x}\,^{-1}\! \cdot\! x|_g^{2r!} + |t-\overline{t}\,|^2\right)\Big) \\&=&  (w^\epsilon - v_\epsilon)(x_\epsilon^\sigma,t_\epsilon^\sigma)- \frac{1}{\sigma} \left(|\overline{x}\,^{-1}\! \cdot\! x_\epsilon^\sigma|_g^{2r!} + |t_\epsilon^\sigma-\overline{t}\,|^2\right)\\&\geq& \frac{\alpha}{2} > 0
\end{eqnarray*}
where $(x_\epsilon^\sigma, t_\epsilon^\sigma) \in B(0,R) \times (0,T).$  Now it follows from the fact that $w^\epsilon - v_\epsilon$ is Lipschitz continuous and $(w^\epsilon - v_\epsilon) \to (w-v)$ pointwise that $(x_\epsilon^\sigma, t_\epsilon^\sigma) \to (\overline{x},\overline{t}\,)$ as $\epsilon \to 0$.  Combining this with the fact that $B(0,R)_{(1+2R_0)\epsilon} \times [0,T)_{(1+2R_0)\epsilon} \to B(0,R) \times (0,T)$ as $\epsilon \to 0$ yields the existence of $\epsilon$ small enough that $(x_\epsilon^\sigma, t_\epsilon^\sigma) \in B(0,R)_{(1+2R_0)\epsilon} \times [0,T)_{(1+2R_0)\epsilon}$.   Fix such an $\epsilon$.  By the same argument as the one showing $w^\epsilon$ is semi-convex, we see that 
$$w^\epsilon(x,t) - v_\epsilon(x,t) - \frac{1}{\sigma} (|\overline{x}\,^{-1}x|^{2r!}_g + |t-\overline{t}\,|^2)$$ is semi-convex for each $\sigma$ and each $\epsilon$.  By Jensen's maximum principle $\footnote{Jensen's maximum principle gives $a_{l}^{\sigma,k}$ as a vector in the canonical basis.  However, since we are on a bounded domain, by changing the basis and rescaling, we obtain the $a_{l}^{\sigma,k}$ given.}$ (see \cite[Lemma 3.10]{jensen}, \cite[Lemma 3.4]{CGG}), there exist a sequence $\{(y_l^{\sigma,\,\,\epsilon},s_l^{\sigma,\,\,\epsilon})\}$ such that $(y_l^{\sigma,\,\,\epsilon},s_l^{\sigma,\,\,\epsilon}) \to (x_\epsilon^\sigma,t_\epsilon^\sigma)$ as $l \to \infty$ and a sequence $\{(a_l^{\sigma,\,\,\epsilon},b_l^{\sigma,\,\,\epsilon})\}$ satisfying $|a_l^{\sigma,\,\,\epsilon}|_g \leq \frac{1}{l}$, $|b_l^{\sigma,\,\,\epsilon}| \leq \frac{1}{l}$ such that 
\begin{equation}\label{newmax}
w^\epsilon(x,t) - v_\epsilon(x,t) - \frac{1}{\sigma}(|\overline{x}\,^{-1}x|^{2r!}_g + |t-\overline{t}\,|^2) + <a_l^{\sigma,\,\,\epsilon},x>_g + b_l^{\sigma,\,\,\epsilon}t
\end{equation} attains a maximum at $(y_l^{\sigma,\,\,\epsilon},s_l^{\sigma,\,\,\epsilon})$ and is twice differentiable (in the Euclidean sense) at $(y_l^{\sigma,\,\,\epsilon},s_l^{\sigma,\,\,\epsilon})$.  Therefore,
$$Dw^\epsilon(y_l^{\sigma,\,\,\epsilon},s_l^{\sigma,\,\,\epsilon}) = Dv_\epsilon(y_l^{\sigma,\,\,\epsilon},s_l^{\sigma,\,\,\epsilon}) + \frac{1}{\sigma}D(|\overline{x}\,^{-1}y_l^{\sigma,\,\,\epsilon}|^{2r!}_g) + D(<a_l^{\sigma,\,\,\epsilon},y_l^{\sigma,\,\,\epsilon}>_g),$$
$$w^\epsilon_t(y_l^{\sigma,\,\,\epsilon},s_l^{\sigma,\,\,\epsilon}) = (v_\epsilon)_t (y_l^{\sigma,\,\,\epsilon},s_l^{\sigma,\,\,\epsilon}) - b_l^{\sigma,\,\,\epsilon} + \frac{2}{\sigma}(s_l^{\sigma,\,\,\epsilon} - \overline{t}\,),$$
and
$$D^2w^\epsilon(y_l^{\sigma,\,\,\epsilon},s_l^{\sigma,\,\,\epsilon}) \leq D^2v_\epsilon(y_l^{\sigma,\,\,\epsilon},s_l^{\sigma,\,\,\epsilon}) + \frac{1}{\sigma}D^2(|\overline{x}\,^{-1}y_l^{\sigma,\,\,\epsilon}|^{2r!}_g).$$
Then by \cite[Lemma 5.4]{CC}
$$D_0w^\epsilon(y_l^{\sigma,\,\,\epsilon},s_l^{\sigma,\,\,\epsilon}) = D_0v_\epsilon(y_l^{\sigma,\,\,\epsilon},s_l^{\sigma,\,\,\epsilon}) + \frac{1}{\sigma}D_0(|\overline{x}\,^{-1}y_l^{\sigma,\,\,\epsilon}|^{2r!}_g) + (a_l^{\sigma,\,\,\epsilon})_H,$$
$$w^\epsilon_t(y_l^{\sigma,\,\,\epsilon},s_l^{\sigma,\,\,\epsilon}) = (v_\epsilon)_t (y_l^{\sigma,\,\,\epsilon},s_l^{\sigma,\,\,\epsilon}) - b_l^{\sigma,\,\,\epsilon} + \frac{2}{\sigma}(s_l^{\sigma,\,\,\epsilon} - \overline{t}\,),$$ 
and
$$\left(D^2_0w^\epsilon(y_l^{\sigma,\,\,\epsilon},s_l^{\sigma,\,\,\epsilon})\right)^* \leq \left(D^2_0v_\epsilon(y_l^{\sigma,\,\,\epsilon},s_l^{\sigma,\,\,\epsilon})\right)^* + \frac{1}{\sigma}\left(D^2_0(|\overline{x}\,^{-1}y_l^{\sigma,\,\,\epsilon}|^{2r!}_g)\right)^*.$$
By the semiconvexity of $w^\epsilon$ and $-v_\epsilon$ and Jensen's maximum principle we have 
\begin{eqnarray*}
\frac{-1}{\epsilon} I &\leq& D^2w^\epsilon(y_l^{\sigma,\,\,\epsilon},s_l^{\sigma,\,\,\epsilon}) \\&\leq& D^2v_\epsilon(y_l^{\sigma,\,\,\epsilon},s_l^{\sigma,\,\,\epsilon}) + \frac{1}{\sigma} D^2(|\overline{x}\,^{-1}y_l^{\sigma,\,\,\epsilon}|^{2r!}) \\&\leq& \frac{1}{\epsilon}I + \frac{1}{\sigma}D^2(|\overline{x}\,^{-1}y_l^{\sigma,\,\,\epsilon}|^{2r!}) \\&\leq& \frac{1}{\epsilon}I + cI
\end{eqnarray*}
since we are working on a bounded domain.  Therefore, we also have that \\$D_0^2w^\epsilon(y_l^{\sigma,\,\,\epsilon},s_l^{\sigma,\,\,\epsilon})$ and $D_0^2v_\epsilon(y_l^{\sigma,\,\,\epsilon},s_l^{\sigma,\,\,\epsilon})$ are bounded above and below since the change of basis matrix $A$ is smooth and we are on a bounded domain.  Since $w^\epsilon$ is a viscosity subsolution and $v_\epsilon$ is a viscosity supersolution and both are differentiable at $(y_l^{\sigma,\,\,\epsilon},s_l^{\sigma,\,\,\epsilon})$, 
$$w^\epsilon_t(y_l^{\sigma,\,\,\epsilon},s_l^{\sigma,\,\,\epsilon}) + F\left(D_0w^\epsilon(y_l^{\sigma,\,\,\epsilon},s_l^{\sigma,\,\,\epsilon}),\left(D^2_0w^\epsilon(y_l^{\sigma,\,\,\epsilon},s_l^{\sigma,\,\,\epsilon})\right)^*\right) \leq - \delta$$
and
\begin{eqnarray*}
w^\epsilon_t(y_l^{\sigma,\,\,\epsilon},s_l^{\sigma,\,\,\epsilon}) &+& b_l^{\sigma,\,\,\epsilon} - \frac{2}{\sigma}(s_l^{\sigma,\,\,\epsilon} - \overline{t}\,) \\&+& F\Bigg(D_0w^\epsilon(y_l^{\sigma,\,\,\epsilon},s_l^{\sigma,\,\,\epsilon})- \frac{1}{\sigma}D_0(|\overline{x}\,^{-1}y_l^{\sigma,\,\,\epsilon}|^{2r!}) - (a_l^{\sigma,\,\,\epsilon})_H,\\&& \hspace{5pt} \left(D^2_0w^\epsilon(y_l^{\sigma,\,\,\epsilon},s_l^{\sigma,\,\,\epsilon})\right)^* - \frac{1}{\sigma}\left(D_0^2(|\overline{x}\,^{-1}y_l^{\sigma,\,\,\epsilon}|^{2r!})\right)^*\Bigg) \geq 0.
\end{eqnarray*}
Combining these and applying the degenerate ellipticity of $F$
\begin{eqnarray*}
0 < \delta &\leq& w^\epsilon_t(y_l^{\sigma,\,\,\epsilon},s_l^{\sigma,\,\,\epsilon}) + b_l^{\sigma,\,\,\epsilon} - \frac{2}{\sigma}(s_l^{\sigma,\,\,\epsilon} - \overline{t}\,)\\&+& F\Bigg(D_0w^\epsilon(y_l^{\sigma,\,\,\epsilon},s_l^{\sigma,\,\,\epsilon})- \frac{1}{\sigma}D_0(|\overline{x}\,^{-1}y_l^{\sigma,\,\,\epsilon}|^{2r!}) - (a_l^{\sigma,\,\,\epsilon})_H,\\&&  \left(D^2_0w^\epsilon(y_l^{\sigma,\,\,\epsilon},s_l^{\sigma,\,\,\epsilon})\right)^* - \frac{1}{\sigma}\left(D_0^2(|\overline{x}\,^{-1}y_l^{\sigma,\,\,\epsilon}|^{2r!})\right)^*\Bigg) \\ 
 &-& w^\epsilon_t(y_l^{\sigma,\,\,\epsilon},s_l^{\sigma,\,\,\epsilon}) - F\left(D_0w^\epsilon(y_l^{\sigma,\,\,\epsilon},s_l^{\sigma,\,\,\epsilon}),\left(D^2_0w^\epsilon(y_l^{\sigma,\,\,\epsilon},s_l^{\sigma,\,\,\epsilon})\right)^*\right)\\
 &=& F\Bigg(D_0w^\epsilon(y_l^{\sigma,\,\,\epsilon},s_l^{\sigma,\,\,\epsilon})- \frac{1}{\sigma}D_0(|\overline{x}\,^{-1}y_l^{\sigma,\,\,\epsilon}|^{2r!}) - (a_l^{\sigma,\,\,\epsilon})_H,\\&&  \left(D^2_0w^\epsilon(y_l^{\sigma,\,\,\epsilon},s_l^{\sigma,\,\,\epsilon})\right)^* - \frac{1}{\sigma}\left(D_0^2(|\overline{x}\,^{-1}y_l^{\sigma,\,\,\epsilon}|^{2r!})\right)^*\Bigg)\\
 &-& F\left(D_0w^\epsilon(y_l^{\sigma,\,\,\epsilon},s_l^{\sigma,\,\,\epsilon}),\left(D^2_0w^\epsilon(y_l^{\sigma,\,\,\epsilon},s_l^{\sigma,\,\,\epsilon})\right)^*\right)\\&+& b_l^{\sigma,\,\,\epsilon} - \frac{2}{\sigma}(s_l^{\sigma,\,\,\epsilon} - \overline{t}\,)
 \end{eqnarray*}
We then take the limits $l,\,\sigma \to \infty$ of both sides of the inequality.  By the continuity of $F$ we may pass the limits $l, \sigma \to \infty$ inside.  Notice that these limits exist because the corresponding Euclidean derivatives are bounded and the change of basis matrix is bounded since we are working on a bounded domain.  Thus letting $l,\sigma \to \infty$ we have a contradiction.  This yields (i).  \\
(ii)  Taking the limit at $\theta \to 1$ in (i) yields (ii). \\ 
(iii)  The proof of part (iii) follows exactly as in \cite[Theorem 1(b)]{im:main} with the exception that "mollification" is replaced by "left mollification."  
\end{proof}

Notice that if our domain was bounded instead of all of $\mathbb{G}$, we could use the same proof as above without the extra assumptions on $u$ and $v$ to obtain the following:
\begin{cor}\label{comparisongaussbounded}
Let $\G$ be a Carnot group.  Suppose that $\Omega \subset \mathbb{G}$ is a bounded domain and $F:\R^{m_1} \times S^{m_1}(\R) \to \R$ is degenerate elliptic.  Suppose that $u$ is a viscosity subsolution and $v$ is a viscosity supersolution to $\eqref{general}$ on $\Omega \times [0,\infty)$ such that 
$$u(x,0) \leq v(x,0) \hspace{5pt}\forall x \in \Omega$$
and $$u(x,t) \leq v(x,t) \hspace{5pt} \forall (x,t) \in \partial \Omega \times [0,\infty).$$
Then $u \leq v$ on $\Omega \times [0,\infty)$.
\end{cor}

We will now use Perron's Method to construct viscosity solutions to the degenerate parabolic equation
$$u_t + F(t,p,u,D_0u,D_1u,(D_0^2u)^*) = 0.$$

In order to prove the existence of such solutions, we need the following lemmas, the proof of which follow the Euclidean proof in \cite{Giga} with the modifications that Euclidean derivatives are replaced with horizontal derivatives and the Euclidean norm is replaced with the gauge norm.  As such, the proof are omitted here.   
\begin{thm}\label{existence}\textnormal{(Follows as in \cite[Theorem 2.4.3]{Giga})}
Let $\G$ be a Carnot group and $\mathcal{O} \subset \G$.  Assume that $F:[0,T] \times \mathbb{G} \times \R \times \R^{m_1} \times \R^{m_2} \times S^{m_1}(\R) \to \R$ is degenerate elliptic.  Let $h_-$ and $h_+$ be a sub- and supersolution of
\begin{equation}\label{eq}
v_t + F(t,p,v,D_0v,D_1v,(D^2_0v)^*) = 0
\end{equation}
in $\mathcal{O}_T = \mathcal{O} \times (0,T)$, respectively.  If $h_- \leq h_+$ in $\mathcal{O}_T$, then there exists a viscosity solution $u$ of $\eqref{eq}$ that satisfies $h_- \leq u \leq h_+$ in $\mathcal{O}_T$.  
\end{thm}

Notice that Perron's method does not require the use of the comparison principle in order to obtain the existence of a viscosity solution to 
\begin{equation}\label{eq1}
v_t + F(t,p,v,D_0v,D_1v,(D^2_0v)^*) = 0.
\end{equation}
However, in order to show the continuity or uniqueness of such solutions, we will need the comparison principle.  When $\mathcal{O}$ is bounded, we  use the comparison principle given by Corollary \ref{comparisongaussbounded}. 
\begin{thm}
Let $\G$ be a Carnot group and $\mathcal{O} \subset \G$ be bounded.  Suppose $F: \R^{m_1} \times S^{m_1}(\R) \to \R$ is degenerate elliptic.  Suppose $f$ and $g$ are sub- and supersolutions of $$v_t + F(D_0v,(D^2_0v)^*) = 0,$$ respectively, in $\mathcal{O}_T = \mathcal{O} \times (0,T)$ satisfying $f \leq g$ on $\mathcal{O}_T$ and $f_* =g^*$ on $\partial\mathcal{O}_T = (\mathcal{O}\times \{0\}) \cup(\partial\mathcal{O} \times [0,T))$.  Then there is a viscosity solution $u$ of $\eqref{eq1}$ satisfying $u \in C(\overline{\mathcal{O}}_T)$ and $f \leq u \leq g$ on $\bar{\mathcal{O}}_T$. 
\end{thm}

\begin{thm}
Let $\G$ be a Carnot group, $\mathcal{O} \subset \G$ be bounded, and $\mathcal{O}_T  = \mathcal{O} \times (0,T)$.  Suppose $F: \R^{m_1} \times S^{m_1}(\R) \to \R$ is degenerate elliptic.  For given $g \in C(\partial \mathcal{O}_T)$ there is at most one solution $u$ of $$v_t + F(D_0v,(D^2_0v)^*) = 0$$ in $\mathcal{O}_T$ with $u = g$ on $\partial \mathcal{O}_T$. 
\end{thm}

In order to use the comparison principle for an unbounded domain, Theorem \ref{comparisongauss}, we need to have the existence of a function $h_0(x)$ satisfying the hypothesis.  Because of the strong relationship between the sub- and supersolutions and this $h_0$, we save the existence of a continuous viscosity solution to $\eqref{general}$ on $\mathbb{G} \times [0,\infty)$ satisfying given initial conditions for the specific example of the horizontal Gauss curvature flow equation given in the next section.

\section{Application to the Horizontal Gauss Curvature Flow Equation}
Our goal is to prove the existence of viscosity solutions to the parabolic equation which describes the horizontal Gauss curvature flow of the graph of a continuous function $u$.  To do this, we need to introduce some more notations and develop a formula for the horizontal Gauss curvature of such a surface.  These ideas were first developed by D. Danielli, N. Garofalo, and D.-M. Nhieu in \cite{dgn} and studied further by L. Capogna, S. Pauls, and J. Tyson in \cite{cpt}, R. Hladky and S. Pauls in \cite{hp:mean}, Danielli, Garofalo, and Nhieu in \cite{dgn3}, and Selby in \cite{selby}.  In the next section we follow \cite{cpt} in the development of the definition of the horizontal second fundamental form from which follows the definition of the horizontal Gauss curvature. 
\subsection{Hypersurfaces in Carnot Groups and Horizontal Gauss Curvature}
For each $L > 0$ we define Riemannian metrics $g_L$, the anisotropic dilations of the metric $g$, characterized by $g_L(X_{i,1},X_{j,1}) = g(X_{i,1},X_{j,1}) = \delta_{ij}$, $g_L(X_{i,1},X_{j,k}) = g_L(X_{i,1},X_{j,k}) = 0$ for all $k \neq 1$ and for all $i,j$, and $g_L(X_{i,j},X_{k,l}) = L^{2/l}\delta_{ik}\delta_{jl}$ for all $j,l \neq 1$.  We define a new, rescaled frame which is orthonormal with respect to $g_L$:
$$\mathcal{F}_1^{\mathbb{G}} = \left\{ X_{1}, \ldots, X_{m_1}, \tilde{X}_{1,2}, \ldots , \tilde{X}_{m_r,r}\right\}$$
where $X_{i} = X_{i,1}$ and $\tilde{X}_{i,j} = L^{-1/j}X_{i,j}$ for $2 \leq j \leq r$.  

Let $M$ be a smooth hypersurface in $\mathbb{G}$ given by 
$$M = \{ x \in \mathbb{G} : u(x) = 0\}$$
where $u: \mathbb{G} \to \R$ is a smooth function.  Denote the characteristic set of $M$ by $\Sigma(M) = \{x \in M : H_x\mathbb{G} \subset T_xM\}$ where $H_x\mathbb{G}$ is the horizontal space of $\mathbb{G}$ at $x$ and $T_xM$ is the tangent space to $M$ at $x$.  
\setcounter{thm}{8}
\begin{defn}
For any non-characteristic point, the unit horizontal normal to $M$ is defined as the normalized projection of the Riemannian normal to the horizontal subbundle.  In the $\mathcal{F}_1^{\mathbb{G}}$ frame it is given as
$$ \nu_0 = \frac{(X_1)uX_1 + \cdots +(X_{m_1}u)}{|(X_1)uX_1 + \cdots +(X_{m_1}u)|}.$$
\end{defn}
Letting $$\nu_L = \frac{D_Lu}{|D_Lu|}$$ denote the Riemannian unit normal (with respect to $g_L$), we note that $$\lim_{L \to \infty} \nu_L = \nu_0$$ uniformly on compact sets of $M\backslash \Sigma(M)$.  We next consider the basis for $T\mathbb{G}|_M$ adapted to the submanifold M:
$$\mathcal{F}_2^{\mathbb{G}} = \{Z_1, \ldots, Z_{m-1},\nu_L\}$$
where $\{Z_i\}$ is an orthonormal basis for $TM$ in the metric $g_L.$

\begin{defn}
The Riemannian second fundamental form of $M$ in the coordinate frame $\mathcal{F}_2^{\mathbb{G}}$ in $(M,g_L)$ is given by 
$$II_L^{M,\mathcal{F}_2^{\mathbb{G}}} = \left(\left \langle \nabla_{Z_i}\nu_L, Z_j \right \rangle_L \right)_{i,j = 1, \ldots, m-1}$$
where $\nabla$ is the Levi-Civita connection associated to $g_L$.  
\end{defn}

At any non-characteristic point we set 
$$T_0 = \frac{\nu_L - \left \langle \nu_L, \nu_0\right \rangle_L \nu_0}{|\nu_L - \left \langle \nu_L, \nu_0\right \rangle_L \nu_0}|_L,$$
$$a_L = \left \langle \nu_L, \nu_0\right \rangle_L$$
and
$$b_L = \left \langle \nu_L, T_0 \right \rangle_L.$$
Thus $$\nu_L = a_L \nu_0 + b_LT_0.$$
We now define a new basis.
\begin{defn}
Let $(\mathbb{G},g_L)$ be as above and $M$ be a smooth hypersurface in $\mathbb{G}$ given as a level set of a function $u: \mathbb{G} \to \R$.  Then
$$\mathcal{F}_3^{\mathbb{G},g_L} = \left\{ e_0, e_1, \ldots, e_{m_1 - 1}, T_1, \ldots, T_{n-1}, \nu_L\right\}$$
is a basis for $T\mathbb{G}|_{M\backslash \Sigma(M)}$ of the form $\mathcal{F}_2^{\mathbb{G}}$, where 
$$e_0 = b_L\nu_0 - a_LT_0,$$
$\{e_1,\ldots,  e_{m_1-1}\}$ is an orthonormal basis for $HM = TM \cap H\mathbb{G}$, and $\{T_1, \ldots , T_{n-1}\}$ is an orthonormal basis for $VM = TM \cap V\mathbb{G}$.  
\end{defn}

Now we are ready to define the horizontal second fundamental form.

\begin{defn}
Given a smooth hypersurface $M \subset (\mathbb{G}, g_L)$ and the adapted basis $\mathcal{F}_3^{\mathbb{G}}$, we define the horizontal second fundamental form at any non-characteristic point as 
$$II_0^M = \left(\left \langle \nabla_{e_i}\nu_0, e_j\right \rangle_L\right)_{i,j = 1, \ldots, m_1-1}.$$
\end{defn}
Note that the horizontal second fundamental form is not necessarily symmetric.

\begin{defn}
Let $M \subset \mathbb{G}$ be a smooth hypersurface and denote by $(II_0^{M})^*$ its symmetrized horizontal second fundamental form.  The horizontal principal curvatures $k_1, \ldots, k_{m_1-1}$ of $M$ at a point $x \in M$ are the eigenvalues of  $(II_0^{M})^*$.  The horizontal Gauss curvature $G_0^{M}$ of $M$ at $x$ is $\det\left[(II_0^{M})^*\right]$.  
\end{defn}

For a more detailed development of the horizontal Gauss curvature of hypersurfaces in Carnot groups we refer the interested reader to \cite{cpt} and \cite{cdpt}.
%--------------------------------------------------------------------
\subsection{Horizontal Convexity in Carnot Groups}
%--------------------------------------------------------------------
In order to proceed, we need to define an appropriate notion of convexity in Carnot groups.  The theory was first developed by D. Danielli, N. Garofalo, and D.-M Nhieu \cite{dgn}, and independently by G. Lu, J. Manfredi, and B. Stroffolini \cite{lms}, and further studied by L. Capogna, S. Pauls, and J. Tyson in \cite{cpt} and P. Juutinen, G. Lu, J. Manfredi, and B. Stroffolini in \cite{jlms}.    
\begin{defn} \textnormal{(\cite[Definition 5.5]{dgn})}
Let $\G$ be a Carnot group.  A function $u: \G \to \R$ is called weakly H-convex if for any $g \in \G$ and every $\lambda \in [0,1]$ one has 
$$u(g \delta_\lambda (g^{-1}g)) \leq (1-\lambda)u(g) + \lambda u(g') \hspace{5pt} \text{ for every } g' \in H_g\G.$$
\end{defn}
Geometrically this definition gives us that a function $u \in \mathcal{C}^1(\G)$ is weakly H-convex if and only if for every $g \in \mathcal{G}$ the graph of the restriction of $u$ to $H_g\G$ lies above its tangent plane at $\xi_1(g)$ where we have let $g = exp(\xi_1(g) + \cdots + \xi_r(g))$.

Because of the nature of the horizontal Gauss curvature flow problem, we would like a way to describe the weak H-convexity of $u$ in terms of its horizontal Hessian.
\setcounter{thm}{6}
\begin{thm}\textnormal{\cite[Thm. 5.12]{dgn}}
A function $u \in \mathcal{C}^2(\G)$ is weakly H-convex if and only if $(D_0^2u)^*$ is positive semi-definite for every $g \in \G$.
\end{thm}

Since we will be working with the graph of $u$ in $\G \times \R$, we will need the following theorem as well:
\begin{thm}\textnormal{\cite[Cor. 4.4]{cpt}}
Let $\G$ be a Carnot group.  A smooth function $u: \G \to \R$ is weakly H-convex if and only if $(II_0^{\mathcal{G}(u)})^*$ is positive semi-definite where $\mathcal{G}(u) = \{(x,s) \in \G \times \R : u(x) - s = 0\}$ is the epigraph of $u$.  Similarly, $u$ is strictly weakly H-convex if and only if $(II_0^{\mathcal{G}(u)})^*$ is positive definite.
\end{thm}
This gives us several different ways of determining the weak H-convexity of the function $u$.  We would also like to understand the weak H-convexity of a set.  Further, as $u$ yields the description of the surface in terms of level sets, we would like to relate the weak H-convexity of the set given by the epigraph of $u$ with the weak H-convexity of $u$.
\setcounter{thm}{14}
\begin{defn}\textnormal{\cite[Definition 7.1]{dgn}}
A subset $A$ of a Carnot group $\G$ is called weakly H-convex if for any $g \in A$ and every $g' \in A\cap H_g\G$ one has $g_\lambda = g \delta_\lambda (g^{-1} g')  \in A$ for every $\lambda \in [0,1]$.     
\end{defn}
Geometrically, this definition means that $A$ is weakly H-convex if for any $g \in A$, the intersectino of $A$ with the horizontal plane $H_g\G$ is starlike in the Euclidean sense with respect to $g$  at the level of the Lie algebra.  
% A set is starlike with respect to z_0 if whenever a point belongs to the set, then the straight line connecting z_0 to the point is contained in the set.  Any convex set is starlike with respect to each of its points.  See Gamelin pg. 38
With this definition, we have the following proposition:
\setcounter{thm}{8}
\begin{prop}\textnormal{\cite[Proposition 7.6]{dgn}}
Let $\G$ be a Carnot group.  Given a weakly H-convex subset $A \subset \G$, a function $u : A \to \R$ is weakly H-convex if and only if $\textnormal{epi } u = \{(x,s) \in \G \times \R : u(x) \leq s\}$  is a weakly H-convex subset of $\G \times \R$.   
\end{prop}
%-----------------------------------------------------------------------\
\subsection{Evolution of Surfaces in $\G$}
%-----------------------------------------------------------------------

Let $\G$ be a Carnot group.  If $u$ is a smooth function in $x$, for each time $t$ we consider the hypersurface 
$$M_t = \{x \in \G : u(x,t) = 0\}.$$ 

Let $G_0^{M_t}(x)$ denote the horizontal Gauss curvature of $M_t$ at $x$ given by 
$$G_0^{M_t}(x) = \det \left[ (II_0^{M_t})^*(x)\right].$$
Fix $t \geq 0$ and consider $x \in M_t$.  Then the evolution of the point $x$ is given by 
\begin{equation*}
\left\{ \begin{array}{l}
\dot{x}(\sigma) = - G_0^{M_\sigma}(x(\sigma))\nu_0(x(\sigma),\sigma)\\
x(t) = x.
\end{array}
\right.
\end{equation*}
Since $x(\sigma) \in M_\sigma$ for some $\sigma \geq t$ we have 
$$u(x(\sigma), \sigma)= 0.$$
Therefore
$$\left \langle Du(x(\sigma),\sigma),\dot{x}(\sigma) \right \rangle + u_\sigma(x(\sigma),\sigma) = 0,$$
i.e.
$$\left \langle Du(x(\sigma), \sigma), -G_0^{M_\sigma}(x(\sigma))\nu_0(x(\sigma),\sigma)\right \rangle + u_\sigma(x(\sigma),\sigma) = 0$$
where $Du$ is the Euclidean derivative of $u$.  If $X_i = \sum_{j=1}^m a_{i,j}\partial_{x_j}$, we define 
\begin{equation*}
A = \left(\begin{array}{ccc} a_{1,1} & \hdots & a_{1,m} \\ \vdots & \ddots & \vdots \\ a_{m_1,1} & \hdots & a_{m_1,m} \end{array} \right).
\end{equation*}
Then,
\begin{eqnarray*}
\Bigg \langle Du(x(\sigma), \sigma), -G_0^{M_\sigma}(x(\sigma)) \frac{A^TA(Du)}{|A(Du)|} \Bigg \rangle &=&
\frac{-G_0^{M_\sigma}(x(\sigma))}{|A(Du)|} \left\langle A(Du),A(Du)\right \rangle\\
&=& -G_0^{M_\sigma}(x(\sigma))\frac{|A(Du)|^2}{|A(Du)|} \\
&=& -G_0^{M_\sigma}(x(\sigma))|D_0u|.
\end{eqnarray*} 
Therefore,
\begin{eqnarray}\label{usigma}
\nonumber u_\sigma(x(\sigma),\sigma) &=& G_0^{M_\sigma}(x(\sigma))|D_0u|\\
&=& \det \left[(II_0^{M_\sigma})^*\right]|D_0u|
\end{eqnarray}
Now we would like to have an explicit description of $G_0^{M_\sigma}$ in terms of $D_0u$ and $(D_0^2u)^*$.  From \cite{cpt} we have the following proposition.

\begin{prop}\textnormal{(\cite[Proposition 3.13]{cpt})}
Let $u: \G \to \R$ be a smooth function and $M = \{u(x) = 0\}$.  Let $A = [\mathcal{F}_3^\G \to \mathcal{F}_1^\G]$ be the change of basis matrix from $\mathcal{F}_3^\G$ to $\mathcal{F}_1^\G$.  We have the following identity of bilinear forms:
$$|D_0u|(II_0^M)^* = (A^T(D_0^2u)^*A)|_{HM}$$
at non-characteristic points.  Here $HM = TM \cap H\G$ as before.
\end{prop}
Therefore,
$$\det ((II_0^{M_t})^*) = \det \left(\frac{1}{|D_0u|}(D_0^2u)^*|_{HM_t}\right).$$
In order to have a "nicer" expression for the above determinant, we need to better understand the restriction of $(D_0^2u)^*$ to $HM_t$.  
Notice that $(D_0^2u)^*$ is a bilinear form on $H\G|_{M_t}$, a series of standard computations yields

\begin{eqnarray*}
\det\!\!\!\!\!\!\!&&\!\!\!\!\!\!\!\Big(\frac{1}{|D_0u|} (D_0^2u)^*|_{HM_t}\Big) \\&=& \det \left(\frac{1}{|D_0u|} (I_{m_1} - \nu_0\otimes\nu_0)(D_0^2u)^*(I_{m_1} - \nu_0\otimes \nu_0) + \nu_0 \otimes \nu_0\right)
\end{eqnarray*}

Combining this with $\eqref{usigma}$ we obtain the following equation describing the horizontal Gauss curvature flow of the original surface $M_0$:
\begin{equation}\label{gaussgenerallevelset}
u_t = |D_0u|\det \left(\frac{1}{|D_0u|} (I_{m_1} - \nu_0\otimes\nu_0)(D_0^2u)^*(I_{m_1} - \nu_0\otimes \nu_0) + \nu_0 \otimes \nu_0\right).
\end{equation}
\setcounter{thm}{3}
\begin{example} \textbf{Self-similarly shrinking cylinder}
Let $\G$ be a Carnot group of step r and for $R_0 > 0$ let 
\begin{equation}\label{cylinder}
u(x,t) = |x_H|^2 - (-m_1t + R_0^{m_1})^{2/m_1}
\end{equation}
Then the level sets 
$$M_t = \{x : u(x,t) = 0\}$$
$$M_0 = \{x : u(x,0) = 0\}.$$
are products of a sphere in $V^1$ with $V^2 \oplus \cdots \oplus \ V^r$.  
Notice that the function's spatial term only depends on variables from the first layer.  Thus its horizontal Gauss curvature reduces to the Euclidean Gauss curvature in $\R^{m_1}$.  This yields that $u$ is a solution to $\eqref{gaussgenerallevelset}$ away from those points at which $D_0u=0$, known as characteristic points.   However, $M_t$ does not contain any characteristic points.  In particular, we see that 
$$u_t = 2(-m_1t + R_0^{m_1})^{2/m_1 - 1}$$
and 
\begin{eqnarray*}
|D_0u|\det[(II_0^{M_t})^*] &=& 2(-m_1 + R_0^{m_1})^{1/m_1} \cdot \frac{1}{(-m_1 + R_0^{m_1})^{\frac{m_1-1}{m_1}}}\\
&=&2(-m_1t + R_0^{m_1})^{2/m_1 - 1}.
\end{eqnarray*}
Finally, observe that $M_t$ gives a self-similar flow of $$M_0 = \{x : u_0(x) = |x_H|^2 - R_0^2 = 0\}.$$  In particular, $M_t = \delta_{\lambda(t)}M_0$ where $\lambda(t) =\frac{(-m_1t + R_0^{m_1})^{1/m_1}}{R_0}.$
\end{example}
Notice, that $\eqref{gaussgenerallevelset}$ has a singularity whenever $D_0u=0$, i.e. at characteristic points of $M_t$.  Even in the Euclidean setting, in which the equation has the same form except $D_0u$ is replaced by $Du$ and $(D_0^2u)^*$ is replaced by $D^2u$, this singularity poses significant difficulties.  In particular, we first notice that the definition of viscosity solution as it is stated makes no sense when $D_0u = 0$.  In the Euclidean setting, an extended definition of viscosity solution is used (see \cite[Section 2.1.3]{Giga}, \cite{ish:sou}, \cite{im:level}).  However, for Carnot group is it still unclear what the appropriate extension should be.  Because of this, we will restrict ourselves to the case when $M_t$ is guaranteed to have no characteristic points, i.e. when $M_t$ is a graph.  
%---------------------------------------------------------------------
\subsection{Evolution of Graphs in $\G \times \R$}
%---------------------------------------------------------------------
Consider the Carnot group $\mathbb{G} \times \R$ with coordinates $(x,s)$, $x \in \mathbb{G}$, $s \in \R$.  On the level of the Lie algebra, this corresponds to adding a single vector field, denoted $S = \frac{\partial}{\partial s}$, to the first layer of the grading.  If $u$ is a smooth function in $x$, for each time $t$ we consider the graph: 
$$\mathcal{G}_t(u)  =\{(x,s) \in \mathbb{G} \times \R : u(x,t) - s = 0\}.$$
The unit horizontal normal $\nu_0$ to $\mathcal{G}_t(u)$ is given by 
\begin{equation}\label{nu}
\nu_0 = \frac{D_0 u - S}{\sqrt{1 + |D_0u|^2}}.
\end{equation}
It is given in \cite[Theorem 4.3]{cpt} that the horizontal Gauss curvature of $\mathcal{G}_t(u)$ is given by 
$$\det\left[(II_0^{\mathcal{G}_t(u)})^*\right]= \frac{\det \left[ (D_0^2u)^*\right]}{\left(\sqrt{1+ |D_0u|^2}\right)^{m_1+2}}.$$
Following the same development as in the previous subsection we obtain
$$u_t = \sqrt{1 + |D_0u|^2}\det \left[(II_0^{\mathcal{G}_t(u)})^*\right] $$
Thus the equation describing the horizontal Gauss curvature flow of the graph of $u$ is given by 
\begin{equation}\label{Gauss}
u_t = \frac{\det((D_0^2u)^*)}{(\sqrt{1+|D_0u|^2})^{m_1+1}}.
\end{equation}
In order to apply Theorem \ref{comparisongauss} to $\eqref{Gauss}$, it is necessary that 
$$F(D_0u,(D_0^2u)^*) = -\frac{\det((D_0^2u)^*)}{\left(\sqrt{1+|D_0u|^2}\right)^{m_1+1}}$$ be continuous, degenerate elliptic and satisfy the property that if $u(x,t)$ is a viscosity subsolution then $\mu u(x,\theta t )$ is a viscosity subsolution for $\theta \in (0,1)$ and $\mu \in (0,1)$ satisfying $\theta \mu^{-(m_1 - 1)} \leq 1$.  Since the continuity of $F$ is clear, our first problem is degenerate ellipticity.  However, equation $\eqref{Gauss}$ does not satisfy this condition in general.  To remedy this, we introduce a new problem for which degenerate ellipticity does hold.  \\
For any $X \in S^{m_1}(\R)$, define
$$\text{det}_+ X = \prod_{i=1}^{m_1} \max\{\lambda_i,0\}$$
where $\{\lambda_i\}$ denotes the eigenvalues of $X$.  We then redefine the function
$$F(D_0u,(D_0^2u)^*) = -\frac{\text{det}_+((D_0^2u)^*)}{\left(\sqrt{1+|D_0u|^2}\right)^{m_1+1}}$$ and consider the problem
\begin{equation}\label{Gaussplus}
u_t = \frac{\text{det}_+((D_0^2u)^*)}{\left(\sqrt{1+|D_0u|^2}\right)^{m_1+1}}
\end{equation}
Recall that in the Euclidean setting, the Gauss curvature flow preserves the strict convexity of the original surface.  This yields the equivalence of the modified and original problems as long as the original surface is strictly convex.  In general, the proof relies heavily on the comparison principle as we must obtain information concerning $u(x,t)$ for $t > 0$ from $u(x,0)$.  For the setting of Carnot groups, we begin with the corresponding theorem pertaining to functions $u: \Omega \times [0,T) \to \R$ where $\Omega \subset \G$ is bounded.
\setcounter{thm}{10}
\begin{thm}\label{convexitybounded}
Let $\G$ be a Carnot group.  Suppose $u: \Omega \times [0,T) \to \R$, where $\Omega \subset \G$ is bounded, is a smooth solution to 
\begin{equation}
\left\{ \begin{array}{lr} 
u_t = \frac{\det_+ ((D_0^2u)^*)}{(\sqrt{1+|D_0u|^2})^{m_1+1}} & (*)\\
u(x,0) = u_0(x)\text{ on } \overline{\Omega}&\\
u(x,t) = g(x,t) \text{ on } \partial \Omega \times [0,T)&
\end{array}\right.
\end{equation}
where $g(x,t)$ is such that $g_t(x,t) \geq \delta > 0$ for all $(x,t) \in \partial \Omega \times [0,T)$ and $u_0(x)$ is strictly weakly H-convex.  Then $u(x,t)$ is strictly weakly H-convex for each fixed $t \geq 0$.
\end{thm}
\begin{proof}
For this proof, we will follow P. Marcati and M. Molinari \cite[Lemma 2.4]{marcati}.  First we differentiate (*) with respect to $t$.  Using Jacobi's formula for the derivative of a determinant we get
\begin{eqnarray*}
u_{tt} &=& \frac{-(m_1+1)\text{det}_+((D_0^2u)^*)}{(1+|D_0u|^2)^{\frac{m_1+1}{2}+1}} \left(\sum_{i=1}^{m_1}X_iuX_iu_t\right) \\&&+ \frac{\det_+ ((D_0^2u)^*)}{(\sqrt{1+|D_0u|^2})^{m_1+1}} \sum_{i,j=1}^{m_1} \left[((D_0^2u)^*)^{-1}\right]_{ji} \left[(D_0^2u_t)^*\right]_{ij}.
\end{eqnarray*}
Letting $u_t = v$ we rewrite the above equation.
\begin{eqnarray}\label{veqn}
\nonumber v_{t} &=& \frac{-(m_1+1)\text{det}_+((D_0^2u)^*)}{(1+|D_0u|^2)^{\frac{m_1+1}{2}+1}} \left(\sum_{i=1}^{m_1}X_iuX_iv\right) \\&&+ \frac{\det_+ ((D_0^2u)^*)}{(\sqrt{1+|D_0u|^2})^{m_1+1}}  \sum_{i,j=1}^{m_1} \left[((D_0^2u)^*)^{-1}\right]_{ji} \left[(D_0^2v)^*\right]_{ij}.
\end{eqnarray}
Our goal is to apply the viscosity theory of the previous section to 
\begin{eqnarray*}
\tilde{F}(\eta, M) &=& \frac{(m_1+1)\text{det}_+((D_0^2u)^*)}{(1+|D_0u|^2)^{\frac{m_1+1}{2}+1}} \left(\sum_{i=1}^{m_1}X_iu\cdot\eta_i\right) \\&&- \frac{\det_+ ((D_0^2u)^*)}{(\sqrt{1+|D_0u|^2})^{m_1+1}}  \sum_{i,j=1}^{m_1} \left[((D_0^2u)^*)^{-1}\right]_{ji} M_{ij}.
\end{eqnarray*}
In order to do so, we must have that $\tilde{F}$ is degenerate elliptic.  Let $M, N \in S^{m_1}(\R)$ such that $M \leq N$.  Notice that this is equivalent to requiring that $N - M$ is positive semi-definite.  We want to show that $\tilde{F}(\xi, N) \leq \tilde{F}(\xi,M)$.  From $\eqref{veqn}$, 
\begin{eqnarray*}
\tilde{F}(\xi, M) - \tilde{F}(\xi, N) &=& \frac{\det_+ ((D_0^2u)^*)}{(\sqrt{1+|D_0u|^2})^{m_1+1}}  
\cdot \sum_{i,j=1}^{m_1}\left[\left(((D_0^2u)^*)^{-1}\right)^T\right]_{ij}(N-M)_{ij}
\end{eqnarray*}
Recall Fejer's theorem \cite[Corollary 7.5.4]{matrixanalysis}: $\sum_{i,j=1}^{m_1}\left[((D_0^2u)^*)^{-1}\right]_{ji}(N-M)_{ij} \geq 0$ for any positive semi-definite $(N-M)$ if and only if $\left(((D_0^2u)^*)^{-1}\right)^T$ is positive semi-definite.  Noticing that $\tilde{F}(\xi, M) - \tilde{F}(\xi, N) =0$ unless $(D_0^2u)^*$ is positive semi-definite and that $\left(((D_0^2u)^*)^{-1}\right)^T$ is positive semi-definite whenever $(D_0^2u)^*$ is, Fejer's theorem yields $\tilde{F}(\xi, N) \leq \tilde{F}(\xi, M)$ as desired.  

Since $u_0$ is strictly weakly H-convex by assumption, we have that the eigenvalues of $(II_0^{\mathcal{G}_0(u_0)})^*$ are strictly positive.  Therefore,
$$v(x,0) = u_t(x,0) = \sqrt{1+|D_0u|^2} \det((II_0^{\mathcal{G}_0(u_0)})^*) \geq \delta > 0.$$
Further, by our hypothesis, 
$$v(x,t) = u_t(x,t) = g_t(x,t) \geq \delta > 0 \text{ for } 0 \leq t < T, \hspace{3pt} x \in \partial \Omega.$$  
Combining these, $v(x,t) \geq \delta > 0$ on $\left(\partial \Omega \times [0,T)\right) \cup \left(\overline{\Omega} \times \{0\}\right)$.  Therefore, by Corollary \ref{comparisongaussbounded}, $$0 < v(x,t) = u_t(x,t) = \sqrt{1 + |D_0u|^2} \det ((II_0^{\mathcal{G}_t(u)})^*)$$ for $(x,t) \in \Omega \times [0,T)$.  Finally, by the continuity of the eigenvalues, we have that the eigenvalues of $(II_0^{\mathcal{G}_t(u)})^*$ are strictly positive and that $u(x,t)$ is strictly weakly H-convex for all $t$.      
\end{proof}   
In order to extend this proof to $u: \G \times [0,T) \to \R$, we immediately see that the comparison principle in this case relies on the existence of the function $h_0$ described in Theorem \ref{comparisongauss}.   Because of this, we have the following theorem concerning the preservation of convexity for unbounded domains.

\begin{thm}\label{convexityunbounded}
 Suppose that $u$ is a smooth solution to 
\begin{equation*}
 \left\{ \begin{array}{lr}
 u_t = \frac{\det_+((D_0^2u)^*)}{(\sqrt{1+|D_0u|^2})^{m_1+1}} & (*) \\
 u(x,0) = u_0(x)
 \end{array} \right.
 \end{equation*} such that $u_0$ is strictly weakly H-convex and 
 \begin{eqnarray*}
\!\!\!\!\!&&\!\!\!\!\!\frac{-(m_1+1)}{1+|D_0u_0|}G(D_0u_0, (D_0^2u_0)^*) \left(\sum_{i=1}^{m_1}X_iu_0X_iG(D_0u_0, (D_0^2u_0)^*)\right) \\&+& G(D_0u_0, (D_0^2u_0)^*) \sum_{i,j=1}^{m_1} \left[((D_0^2u_0)^*)^{-1}\right]_{ji} \left[\left(D_0^2G(D_0u_0, (D_0^2u_0)^*)\right)^*\right]_{ij}
\geq 0 \end{eqnarray*}
where $G(D_0u_0, (D_0^2u_0)^*) = \frac{\det_+((D_0^2u_0)^*)}{(\sqrt{1+|D_0u_0|^2})^{m_1+1}}$.  
Further suppose that there exists \\$h_0(x) \in C(\G)$ such that $h_0(x) \geq \epsilon_0 |x|_g^{2r!}$ for all $x \in \G$ and for some $\epsilon_0 > 0$ and that 
$$\sup_{(x,t) \in \G \times [0,T]} |u_t(x,t) - h_0(x)| < \infty.$$  
Then $u$ is strictly weakly H-convex for all $t$.  
\end{thm}

\begin{proof}
We will again follow the idea of the proof Marcati and Molinari \cite[Lemma 2.4]{marcati}.  Differentiating (*) with respect to time and letting $v = u_t$ we have:
\begin{eqnarray}\label{veqn2}
\nonumber v_{t} &=& \frac{-(m_1+1)\det_+((D_0^2u)^*)}{(1+|D_0u|^2)^{\frac{m_1+1}{2}+1}} \left(\sum_{i=1}^{m_1}X_iuX_iv\right) \\&&+ \frac{\det_+ ((D_0^2u)^*)}{(\sqrt{1+|D_0u|^2})^{m_1+1}}  \sum_{i,j=1}^{m_1} \left[((D_0^2u)^*)^{-1}\right]_{ji} \left[(D_0^2v)^*\right]_{ij}
\end{eqnarray}
As in Theorem \ref{convexitybounded}, we can show that 
\begin{eqnarray*}
\tilde{F}(\eta, M) &=& \frac{(m_1+1)\det_+((D_0^2u)^*)}{(1+|D_0u|^2)^{\frac{m_1+1}{2}+1}} \left(\sum_{i=1}^{m_1}X_iu\cdot\eta_i\right) \\&&- \frac{\det_+ ((D_0^2u)^*)}{(\sqrt{1+|D_0u|^2})^{m_1+1}}  \sum_{i,j=1}^{m_1} \left[((D_0^2u)^*)^{-1}\right]_{ji} M_{ij}.
\end{eqnarray*}
is degenerate elliptic.  Further, if $v$ is a subsolution to $\eqref{veqn2}$, so is $\mu v(x,\theta t)$ for all $\mu,\theta \in (0,1)$.  Since $u_0$ is strictly weakly H-convex by assumption, we have that the eigenvalues of $(II_0^{\mathcal{G}_0(u_0)})^*$ are strictly positive. Therefore,
$$v(x,0) = u_t(x,0) = \sqrt{1+|D_0u_0|^2} \det ((II_0^{\mathcal{G}_0(u_0)})^*) \geq \delta > 0.$$
Define $\varphi(x,t) :=v(x,0)$.  Therefore, $\varphi(x,0) = v(x,0)$ and $\varphi(x,t) > 0$ for all $x,t$.  Further, by our hypothesis,
\begin{eqnarray*}
-\tilde{F}(D_0\varphi, (D_0^2\varphi)^*) &=& -\tilde{F}(D_0v, (D_0^2v)^*)|_{t=0}\\
&=& v_t(x,0) \\
&=& u_{tt}(x,0) \\
&=& \frac{-(m_1+1)\det_+((D_0^2u_0)^*)}{(1+|D_0u_0|^2)^{\frac{m_1+1}{2}+1}}\\
&& \cdot \left(\sum_{i=1}^{m_1}X_iu_0X_i\left(\frac{\det_+((D_0^2u_0)^*)}{(\sqrt{1+|D_0u_0|^2})^{m_1+1}}\right)\right) \\&+& \frac{\det_+ ((D_0^2u_0)^*)}{(\sqrt{1+|D_0u_0|^2})^{m_1+1}} \\&&\cdot \sum_{i,j=1}^{m_1} \left[((D_0^2u_0)^*)^{-1}\right]_{ji} \left[\left(D_0^2\left(\frac{\det_+((D_0^2u_0)^*)}{(\sqrt{1+|D_0u_0|^2})^{m_1+1}}\right)\right)^*\right]_{ij}\\
&\geq& 0.
\end{eqnarray*}
Therefore $\varphi$ is a subsolution.  Also by our hypothesis,
$$\sup_{(x,t) \in \G \times [0,T]} \left(|\varphi(x,t) - h_0(x)| + |v(x,t) - h_0(x)|\right) < \infty.$$
By Theorem \ref{comparisongauss}
$$0 < \varphi(x,t) \leq v(x,t) = u_t(x,t) = \sqrt{1 + |D_0u|^2} \det ((II_0^{\mathcal{G}_t(u)})^*)$$ for $(x,t) \in \Omega \times [0,T)$.  Finally, by the continuity of the eigenvalues, we have that the eigenvalues of $(II_0^{\mathcal{G}_t(u)})^*$ are strictly positive and that $u(x,t)$ is strictly weakly H-convex for all $t$.      
\end{proof}
Using either Theorem \ref{convexitybounded} for bounded domains or Theorem \ref{convexityunbounded} for unbounded domains, each with the appropriate hypotheses, we have that the modified problem is equivalent to the original horizontal Gauss curvature flow problem whenever $u_0$ is strictly weakly H-convex.  Thus in this situation it makes sense to apply our viscosity theory to the modified problem which possesses the degenerate elliptic property we desire.  

Finally, it is an easy computation to see that if $u(x,t)$ is a viscosity subsolution to the modified problem, then so is $\mu u(x,\theta t)$
for $\theta \in (0,1)$ and $\mu \in (0,1)$ satisfying $\theta \mu^{-(m_1 - 1)} \leq 1$.  Then using Perron's Method and the comparison principle, we have the following theorems concerning the existence and uniqueness of continuous viscosity solutions.  
\begin{thm}\label{gaussflowexist}
Let $\G$ be a Carnot group.  Let $h \in \mathcal{C}(\mathbb{G})$ be such that 
$$\sup_\mathbb{G}| h(x) - h_0(x)| < \infty$$
and for each $\epsilon \in (0,1)$ there exists a constant $B_\epsilon > 0$ such that 
$$|h(x) -h(\xi)| \leq \epsilon + B_\epsilon h_0(\xi^{-1}x)$$
where $h_0 \in \mathcal{C}^2$ is as in Theorem \ref{comparisongauss} and satisfies 
$$C \geq \frac{\det_+((D_0^2h_0)^*)}{\left(\sqrt{1+|D_0h_0|^2}\right)^{m_1+1}}$$
for some constant $C > 0$ and $h_0(0) = 0$.  Then there is a viscosity solution $u \in C(\mathbb{G} \times [0,\infty))$ of 
\begin{equation}\label{generalinitial}
\left\{ \begin{array}{ll}
u_t = \frac{\det_+((D_0^2u)^*)}{\left(\sqrt{1+|D_0u|^2}\right)^{m_1+1}} & \text{ in } \mathbb{G} \times (0,\infty)\\
u(x,0) = h(x) & \text{ for } x \in \mathbb{G}
\end{array} \right.
\end{equation}
\end{thm}

\begin{proof}
The idea for this proof follows from \cite[Theorem 2.7]{im:level}.  First we will construct a viscosity supersolution of $\eqref{generalinitial}$.  Define
$$w(x,t) = h_0(x) + Ct \hspace{5pt} \text{ for } (x,t) \in \mathbb{G} \times [0,\infty).$$  
Then by our assumptions on $h_0$, $w$ is a viscosity supersolution.  Further, for each $\zeta \in \mathbb{G}$, $w(\zeta^{-1}x,t)$ is also a viscosity supersolution.

Let $\epsilon \in (0,1)$ and define
$$f(x,t) = \inf_{\xi \in \mathbb{G}, \epsilon \in (0,1)}\{ h(\xi) + \epsilon + B_\epsilon w(\xi^{-1}x,t)\}.$$
By construction, each $h(\xi) + \epsilon + B_\epsilon w(\xi^{-1}x,t)$ is a viscosity supersolution. By Lemma $\ref{closedinf}$, $f(x,t)$ is a viscosity supersolution.  Notice that also by construction,
$$f(x,t) \leq h(x) + \epsilon + B_\epsilon Ct$$
and
$$ h(z) - h(\xi) \leq \epsilon + B_\epsilon h_0(\xi^{-1}x) \leq \epsilon + B_\epsilon h_0(\xi^{-1}x) + Ct \hspace{5pt} \Longrightarrow \hspace{5pt} h(z) \leq f(x,t).$$
Therefore
$$h(x) \leq f(x,0) \leq h(x) + \epsilon \hspace{5pt} \forall \epsilon \in (0,1).$$ 
Thus $h(x) = f(x,0)$.  Further, 
\begin{eqnarray*}
\sup_{(x,t) \in \mathbb{G} \times [0,T]} |f(x,t) - h_0(x)| &\leq& \sup_{(x,t) \in \mathbb{G} \times [0,T]}\Big( |f(x,t) - h(x)| + |h(x) - h_0(x)|\Big) \\  &\leq& \sup_{(x,t) \in \mathbb{G} \times [0,T]}\Big( |\epsilon + B_\epsilon Ct| + |h(x) - h_0(x)|\Big) \\&<& \infty.
\end{eqnarray*}
Now to construct a viscosity subsolution we set 
$$z(x,t) = h(x) \hspace{5pt} \forall (x,t) \in \mathbb{G} \times [0,\infty).$$ Let $\varphi \in \mathcal{C}^{2,1}$ be such that $z(x,t) - \varphi(x,t)$ has a local maximum at $(\hat{x},\hat{t})$.  Notice that since $z(x,t)$ is differentiable in $t$ we must have $z_t = \varphi_t = 0$.  Therefore,
$$\varphi_t(\hat{x},\hat{t}) = 0 \leq \frac{\text{det}_+ \left((D_0^2\varphi(\hat{x},\hat{t}))^*\right)}{\left(\sqrt{1+|D_0\varphi(\hat{x},\hat{t})|^2}\right)^{m_1+1}}.$$
Thus we have that $z(x,t)$ is in fact a viscosity subsolution.  Further, $z(x,t)$ satisfies the hypotheses both parts (i) and (ii) of Theorem \ref{comparisongauss} by construction.  \\
By Theorem \ref{existence}, there exists a solution $u$ to $\eqref{Gaussplus}$ such that 
$$z(x,t) \leq u(x,t) \leq f(x,t) \hspace{5pt} \forall (x,t) \in \mathbb{G} \times (0,\infty).$$
This inequality shows that
$$u^* \leq f^* = h^* = h_* \leq u_* \hspace{5pt} \text{ on } \mathbb{G} \times \{0\}.$$
By Theorem \ref{comparisongauss}, $u^* \leq u_*$ on $\mathbb{G} \times [0,\infty)$.  Thus we have that $u \in C(\mathbb{G} \times [0,\infty))$ and $u(x,0) = h(x)$ on $\mathbb{G}$.  
\end{proof}

\begin{thm}
Let $\G$ be a Carnot group of step $r$.  Suppose $h_0 : \G \to \R$ is such that $h_0 \in C(\G)$ and 
$$h_0(x) \geq \epsilon_0|x|_g^{2r!} \hspace{3pt} \forall x \in \G.$$  If $u$ and $v$ are continuous solutions to 
\begin{equation*}
\left\{ \begin{array}{ll}
u_t = \frac{\det_+((D_0^2u)^*)}{\left(\sqrt{1+|D_0u|^2}\right)^{m_1+1}} & \text{ in } \mathbb{G} \times (0,\infty)\\
u(x,0) = h(x) & \text{ for } x \in \mathbb{G}
\end{array} \right.
\end{equation*}
such that for each $T > 0$
$$\sup_{(x,t) \in \G \times [0,T]} \left(|u(x,t) - h_0(x)| + |v(x,t) - h_0(x)|\right) < \infty$$
then $u=v$ on $\G \times (0,\infty)$.
\end{thm}
\begin{proof}
Let $u$ and $v$ be solutions satisfying the hypothesis of the theorem.  Considering $u$ as a subsolution and $v$ as a supersolution, Theorem \ref{comparisongauss} yields $u \leq v$ on $\G \times (0,\infty)$.  Considering $v$ as a subsolution and $u$ as a supersolution, Theorem \ref{comparisongauss} yields $v \leq u$ on $\G \times (0,\infty)$.  Thus $u=v$ on $\G \times (0,\infty)$.  
\end{proof}
Note that if $h(x)$ satisfies the hypothesis of Theorem \ref{gaussflowexist} and $u$ and $v$ are constructed using the methods of Theorem \ref{gaussflowexist}, then $u$ and $v$ satisfy $$\sup_{(x,t) \in \G \times [0,T]} \left(|u(x,t) - h_0(x)| + |v(x,t) - h_0(x)|\right) < \infty$$ by construction.  This is because such solutions $u$ satisfy:
\begin{eqnarray*}
\sup_{(x,t) \in \G \times [0,T]}|u(x,t) - h_0(x)|&=&\sup_{(x,t) \in \G \times [0,T]} |u(x,t) + h(x) - h(x) - h_0(x)|\\
&\leq&\sup_{(x,t) \in \G \times [0,T]}\left( |u(x,t) - h(x)| + |h(x) - h_0(x)|\right)\\
&\leq& \sup_{(x,t) \in \G \times [0,T]}\left(|f(x,t) - h(x)| + |h(x) - h_0(x)|\right)\\
&\leq& \sup_{(x,t) \in \G \times [0,T]}\left(|f(x,t) - h_0(x)| + 2|h(x) - h_0(x)|\right)\\
&<& \infty.
\end{eqnarray*}

\subsection{Example in H-type Groups}
Recall that the above theorems and constructions rely heavily on the existence of a function $h_0: \G \to \R$ having the properties that $h_0 \in C^2(\G)$, $h_0(0) = 0$, 
$$h_0(x) \geq \epsilon_0|x|_g^{2r!} \hspace{3pt} \forall x \in \G$$
and 
$$ C \geq \frac{\text{det}_+ ((D_0^2 h_0(x))^*)}{\left(\sqrt{1 + |D_0h_0(x)|^2}\right)^{m_1+1}}$$
for some $C > 0$.  In this section we will give an explicit example of such an $h_0$ for H-type groups.

Let $\He$ be an H-type group (see Example 1.3) with Lie algebra given by $\text{\gothfamily{h}} = V^1 \oplus V^2$ such that $\{X_1, \ldots, X_{m_1}\}$ forms an orthonormal basis of $V^1$ and $\{Y_1, \ldots, Y_n\}$ forms an orthonormal basis of $V^2$.  For each $x \in \He$, let $v(x) = V^1$ and $z(x) = V^2$ such that $x = exp(v(x) + z(x))$.  
Let 
$$h_0(x)=(|v(x)|^4 + 16|z(x)|^2) \geq |x|_g^4.$$  Notice that $h_0 \in \mathcal{C}^2(\G)$ and $h_0(0)=0$ by construction.  Thus it remains to be shown that there exists $C > 0$ such that 
$$ C \geq \frac{\text{det}_+ ((D_0^2 h_0(x))^*)}{\left(\sqrt{1 + |D_0h_0(x)|^2}\right)^{m_1+1}}$$
With this in mind, we consider the following.  \\
Recall 
$$X_iu(x) = \frac{\partial}{\partial_s} u(xe^{sX_i})|_{s=0}$$
and that in a Carnot group of step two, the Baker-Campbell-Hausdorff formula yields 
$$e^Xe^Y = e^{X+Y+\frac{1}{2}[X,Y]}.$$  
Let $\varphi_j(s) = h_0(xe^{sX_j})$.  Then $X_jh(x) = \varphi'_j(0)$.  Further, by the Baker-Campbell-Hausdorff formula, $v(xe^{sX_j}) = v(x) + sX_j$ and $z(xe^{sX_j}) = z(x) + \frac{1}{2}[v(x), sX_j]$.  This yields
$$\varphi_j(s) = |v(x) + sX_j|^4 + 16|z(x) + \frac{s}{2}[v(x),X_j]|^2.$$
Using only the fact that we are in a step two group,
$$\varphi'_j(0) = 4\left(|v(x)|^2 \langle v(x), X_j\rangle + 4\langle z(x), [v(x),X_j]\rangle\right).$$
Then by the fact that the group is of H-type,
$$\varphi'_j(0) = 4\left(|v(x)|^2 \langle v(x), X_j\rangle + 4\langle J_{z(x)}v(x), X_j\rangle\right).$$
Using the properties $J_{z(x)}v(x)$,
\begin{eqnarray*}
\sum_{j=1}^{m_1} \left(\varphi_j'(0)\right)^2 &=& \sum_{j=1}^{m_1} 16\left(\langle v(x)|v(x)|^2 + 4J_{z(x)}v(x), X_j \rangle \right)^2\\
&=& 16|v(x)|^2 |h_0(x)|^4
\end{eqnarray*}
Therefore,
$$|D_0(h_0(x))|^2 = 16|v(x)|^2 |h_0(x)|^4.$$
Further,
\begin{eqnarray*}
\left(D_0^2(h_0(x))\right)^*_{ij} = 8\langle v(x), X_i\rangle \langle v(x),X_j\rangle + 4|v(x)|^2\delta_{ij} + 8 \langle [v(x),X_j], [v(x),X_i]\rangle.
\end{eqnarray*}
To bound the determinant of $\left(D_0^2(h_0(x))\right)^*$ we will need the following:
\begin{eqnarray*}
\sum_{j=1}^{m_1} |[v(x),X_j]|^2 &=& \sum_{i=1}^n \sum_{j=1}^{m_1} \langle Y_i, [v(x),X_j]\rangle ^2 \\
&=& \sum_{i=1}^n \sum_{j=1}^{m_1} \langle J_{Y_i}v(x), X_j\rangle ^2\\
&=& \sum_{i=1}^n |J_{Y_i}v(x)|^2\\
&=& \sum_{i=1}^n |Y_i|^2|v(x)|^2 \\
&=& n|v(x)|^2
\end{eqnarray*} 
Using this we get
\begin{eqnarray*}
\det \left( (D_0^2(h_0(x))^*\right) &\leq& \prod_{i=1}^{m_1} \sum_{j=1}^{m_1} \left|\left(D_0^2(h_0(x))\right)^*_{ij}\right|\\
&\leq& \prod_{i=1}^{m_1} \Bigg(\sum_{j=1}^{m_1}\Big(8|\langle v(x), X_j\rangle |^2 + 8 |\langle v(x), X_i \rangle |^2  + 8 |\langle [v(x),X_j]|^2 \\&+& 8|[v(x),X_i]|^2 + 4|v(x)|^2\Big)\Bigg)  \\
&\leq& C\prod_{i=1}^{m_1} \left(\sum_{j=1}^{m_1} \left(|\langle v(x), X_j\rangle |^2 + |[v(x), X_j]|^2 +4|v(x)|^2\right)\right) \\
&=& C\prod_{i=1}^{m_1}\left(|v(x)|^2 + m_1|v(x)|^2 + n|v(x)|^2 \right)\\
&=& C(m_1,n)|v(x)|^{2m_1}
\end{eqnarray*}
and
$$|D_0(h_0(x))|^2 = 16|v(x)|^2|h_0(x)|^2 \geq 16|v(x)|^6.$$
Therefore
\begin{eqnarray*}
\frac{\text{det}_+\left((D_0^2(h_0(x)))^*\right)}{\left(\sqrt{1+|D_0(h_0(x))|^2}\right)^{m_1+1}} &\leq& \frac{\text{det}_+\left((D_0^2(h_0(x)))^*\right)}{\left(\sqrt{1+16|v(x)|^6}\right)^{m_1+1}} \\
&\leq& C(m_1,n) \frac{|v(x)|^{2m_1}}{\left(\sqrt{1+16|v(x)|^6}\right)^{m_1+1}} \\
&\leq& \tilde{C}(m_1,n)
\end{eqnarray*} 

Thus for H-type groups, we have the following theorem concerning the existence of continuous solutions to the horizontal Gauss curvature flow equation.
\begin{thm}
Let $\He$ be an H-type group with Lie algebra given by $\text{\gothfamily{h}} = V^1 \oplus V^2$.  For each $x \in \He$, let $v(x) = V^1$ and $z(x) = V^2$ such that $x = exp(v(x) + z(x))$.  Define $h_0(x)=(|v(x)|^4 + 16|z(x)|^2) \geq |x|_g^4.$  Let $h \in C(\G)$ be such that 
$$\sup_{\G} |h(x) - h_0(x)| < \infty$$ 
and for each $\epsilon \in (0,1)$ there exists a constant $B_\epsilon > 0$ such that 
$$|h(x) - h(\xi)| \leq \epsilon + B_\epsilon h_0(\xi^{-1}x).$$
Then there is a viscosity solution $u \in C(\G \times [0,\infty))$ of 
\begin{equation}
\left\{ \begin{array}{ll}
u_t = \frac{\det_+((D_0^2u)^*)}{\left(\sqrt{1+|D_0u|^2}\right)^{m_1+1}} & \text{ in } \mathbb{G} \times (0,\infty)\\
u(x,0) = h(x) & \text{ for } x \in \mathbb{G}
\end{array} \right.
\end{equation}
\end{thm} 

\bibliographystyle{amsplain}
\bibliography{Thesis}

\end{document}